\documentclass[11pt]{article}

\usepackage[nocompress]{cite}
\usepackage{graphicx}
\usepackage[labelfont=bf]{caption}
\usepackage{enumitem}
\usepackage{tikz}
\usetikzlibrary{shapes, arrows.meta, positioning, fit, backgrounds, calc}
\usetikzlibrary{shapes.geometric, arrows.meta, positioning, calc, shadows}
\definecolor{myblue}{RGB}{70, 130, 180}    
\definecolor{myred}{RGB}{205, 92, 92}      
\definecolor{mygray}{RGB}{105, 105, 105}   
\definecolor{mypurple}{RGB}{153, 102, 204} 

\usepackage{amsthm}
 \usepackage{amsmath}
\usepackage{amsfonts}
\usepackage{amssymb, amsfonts, amsmath, nicefrac}
\usepackage{color}
\usepackage{mathtools}

\usepackage{hyperref}
\hypersetup{
    colorlinks=true,
    linkcolor=blue,
    anchorcolor=blue,
    citecolor=blue,
}
\usepackage[capitalise]{cleveref}

\newtheorem{theorem}{Theorem}[section]
\newtheorem{lemma}{Lemma}[section]
\newtheorem{proposition}{Proposition}[section]
\newtheorem{example}{Example}
\newtheorem{remark}{Remark}
\newtheorem{definition}{Definition}[section]
\newtheorem{assumption}{Assumption}[section]

\crefdefaultlabelformat{#2\textup{#1}#3}
\Crefname{assumption}{Assumption}{Assumptions}
\crefname{assumption}{Assumption}{Assumptions}
\Crefrangeformat{assumption}{Assumptions #3\textup{#1}#4--#5\textup{#2}#6}
\setlist[enumerate]{label=\textup{\arabic*.}, nosep}

\newcommand{\sff}{\mathsf{f}}
\def\argmin{\mathop{\rm arg\,min}}
\def\argmax{\mathop{\rm arg\,max}}
\newcommand{\U}{\mathcal{U}}
\newcommand{\HH}{\mathcal{H}}
\newcommand\tsum{\textstyle\sum\nolimits}
\newcommand{\abs}[1]{\left|#1\right|}
\newcommand{\N}{\mathcal{N}}
\newcommand{\dist}{\mathrm{dist}}
\newcommand{\e}{\varepsilon}

\newcommand{\bbr}{\mathbb{R}}
\newcommand{\bbe}{\mathbb{E}}
\newcommand{\bbp}{\mathbb{P}}

\begin{document}

\title{\bf Stochastic Optimal Control with Side Information and Bayesian Learning}
\author{{\bf Johannes Milz}\thanks{Georgia Institute of Technology, Atlanta, Georgia
30332, USA (\texttt{johannes.milz@isye.gatech.edu}).  The research of this
author was partially supported by the National Science Foundation under Award
No.\ DMS-2410944.}
\and 
{\bf Alexander Shapiro}\thanks{Georgia Institute of Technology, Atlanta, Georgia
30332, USA ({\texttt{ashapiro@isye.gatech.edu}}). 
The research of this
author was partially supported by Air Force Office of Scientific Research (AFOSR)
under Grant    FA9550-25-1-0310.}
\and {\bf Enlu Zhou}\thanks{Georgia Institute of Technology, Atlanta, Georgia
30332, USA (\texttt{enlu.zhou@isye.gatech.edu}). The research of this author was partially supported by Air Force Office of Scientific Research (AFOSR)
under Grant    FA9550-25-1-0310 and the National Science Foundation under Award ECCS-2419562.} 
}

\date{July 20, 2026}

\maketitle

\abstract{%
We study infinite-horizon stochastic optimal control problems with observable
side information: a  Markov chain that modulates an unknown
context-conditional randomness distribution. 
Since this distribution is
unknown, we propose a Bayesian reformulation based on a parametric density
model and posterior predictive dynamics, which yields a Bayesian Bellman
equation. We prove posterior consistency under
Markov samples and, under correct specification and identifiability, uniform
convergence of the Bayesian value function.
Finally, we establish Bernstein--von Mises-type asymptotic normality for the
data-driven contextual optimal value.%
}

\section{Introduction}

Contextual optimization addresses decision making under uncertainty with side information (also known as context or covariates) available at the time of decision. In its static (one‑stage) form, the decision maker observes a context and then chooses an action that optimizes the conditional expected performance. To learn the optimal decision rule that maps from an observed context to the optimal action, there are several 
paradigms, including decision rule optimization,  sequential learning and optimization, and integrated learning and optimization (see the recent paper \cite{Sadana2025ContextualOptimizationSurvey} for a comprehensive survey).

Contextual optimization has also been studied in multi‑stage settings where a state evolves over time under a policy and the context may shape rewards and transition kernels from stage to stage. Several existing models instantiate dynamic contextual  optimization. In the contextual multi‑armed bandit, the learner repeatedly observes a context and selects an arm, receiving bandit feedback while seeking low regret (see, e.g., \cite{LangfordZhang2007}). In contextual Markov decision processes (CMDPs), an episode is governed by an observed, typically static context that parameterizes both transitions and rewards; foundational work established this model and studied planning/learning guarantees, while subsequent results provided regret and sample‑complexity bounds under realizability assumptions and oracle‑based regression \cite{Hallak2015CMDP,LevyMansour2023AAAI,DengEtAl2024AISTATS}. When the context evolves within an episode, Dynamic Contextual MDPs (DCMDPs) allow history‑dependent or time‑varying contexts and analyze the algorithm's regret under additional structure (e.g., logistic DCMDPs) \cite{TennenholtzEtAl2023ICML}. A closely related line of research models the context as a finite‑state Markov chain (often called a regime or mode), yielding Markov‑jump/Markov‑switching systems (MJLS) in which mode‑dependent dynamics are controlled via dynamic programming or coupled Riccati equations; classical treatments typically assume the mode transition matrix and noise statistics are both known rather than learned from data (see, e.g., \cite{Costa2005,GuptaMurrayMJLSNotes,Geromel2023MJLS}).
A related stream studies multi-period inventory control under Markov-modulated demand, 
with known demand distribution driven by a Markov chain \cite{Chen2001,Song1993,Arifouglu2004,Arifouglu2010,Chen2024,Malladi2023,Sethi1997}. 

In this paper, we introduce a contextual stochastic optimal control problem, in which (i) the context is observable and evolves according to a Markov chain with known transition probabilities, and (ii) the distribution of the exogenous randomness, conditional on the context, is unknown and must be learned from sequential data. To estimate this conditional distribution we posit a regression model and learn its parameters with a Bayesian approach, and then solve a Bayesian average estimate of the original control problem. 

Compared to  CMDPs and DCMDPs, our formulation departs by imposing Markovian context dynamics (rather than per-episode static or history-dependent context) and our approach hinges on Bayesian learning of the conditional distribution.  Unlike classical MJLS, which presumes known randomness distributions, we learn the context‑conditional randomness distribution from data and integrate Bayesian learning into decision making. We also distinguish our problem formulation from classical adaptive control or reinforcement learning, where control and learning interact with each other and are tightly coupled. In our formulation, the randomness is exogenous,  and realizations of the randomness are directly observable. Therefore, at each stage, we first estimate the randomness distribution and then solve the estimated problem. Conceptually, our approach can be viewed as an extension of the episodic Bayesian optimal control approach in \cite{SZLW2025}  to the new problem of contextual stochastic control; because of the different problem settings, new challenges arise from learning of the context-dependent randomness distribution and the corresponding analysis. These challenges are not trivial and require a considerably deeper technical analysis.

We summarize our contributions as follows:
\begin{itemize}
    \item We introduce a new model for stochastic optimal control with side information (i.e., observed context) and unknown randomness distribution. It provides a modeling framework for a wide range of applications 
         For example, it can address portfolio optimization with side information such as time-varying economic factors.  As the topic of  context-dependent randomness in the {\em static} setting is popular now with many recent publications,   the current paper extends this to the dynamic framework. 
    \item We propose a reformulation of the original problem by incorporating Bayesian learning of the context-conditional randomness  distribution. The Bayesian learning approach sequentially updates the posterior distribution with observed context-randomness data, and the Bayesian average estimate of the cost function serves as a surrogate of the original (unknown) objective.
    \item We show that the optimal value functions of the reformulated problem converge uniformly to the true value function as the data size increases. We also develop a  Bernstein--von Mises-type asymptotic normality result for the data-driven contextual optimal value.
\end{itemize}

\section{Problem Statement}

We consider the  (discrete time, infinite
horizon)  Stochastic Optimal Control  (SOC) model:
\begin{equation}\label{soc-inf}
\min\limits_{\pi\in \Pi}  \bbe^\pi\big [ \tsum_{t=1}^{\infty}
\gamma^{t-1} c(x_t,u_t,\xi_t)
\big],
\end{equation}
where $\gamma\in (0,1)$ is the discount factor and $\Pi$ is the set of policies
given by
\begin{equation}\label{soc-inf-b}
\Pi=\Big\{\pi=(\pi_1,\pi_2,\ldots):
u_t=\pi_t(x_t,\eta_{[t]},\xi_{[t-1]}),
u_t\in \U, x_{t+1}=F(x_t,u_t,\xi_t),\;\;t=1,2,\ldots\Big\}.
\end{equation}
Here variables  $x_t\in \bbr^{n}$, $t=1,2,\ldots$, represent the state  of the
system with state space $\mathcal{X} \subset \mathbb{R}^n$,  
 $u_t\in \bbr^{m}$,  $t=1,2,\ldots$, are controls,   $\xi_t\in \Xi$,
$t=1,2,\ldots$, are random vectors, and $\eta_t\in \HH$, $t=1,2,\ldots$, are
observable context variables taking values in a {\em  finite} set $\HH$.
The set $\Xi$ is a closed subset of $\bbr^{d}$, the one-stage cost function
$c: \bbr^n\times\bbr^{m}\times\bbr^{d}\to \bbr$ is measurable, and
$F: \bbr^n \times\bbr^{m}\times\bbr^{d}\to \bbr^n$ is a measurable mapping,
with $F(\mathcal{X} \times \mathcal{U} \times \Xi) \subset \mathcal{X}$.
The control set $\U$ is a nonempty subset of $\bbr^{m}$.

We  assume that   $x_1$, $\xi_0$,   and $\eta_1$ are deterministic, 
and that 
the cost function $c \colon \mathcal{X} \times \mathcal{U} \times \Xi \to \mathbb{R}$ is bounded.
Moreover, we assume that the probability law of the random
process $\{(\xi_t,\eta_t)\}_{t\ge 1}$ does not depend on the decisions.
The optimization in \eqref{soc-inf} is performed over policies $\pi\in \Pi$
determined by decisions $u_t$ and state variables $x_t$, which are functions
of the context histories $\eta_{[t]}\coloneqq   (\eta_1,\ldots,\eta_t)$ and randomness histories
$\xi_{[t-1]}\coloneqq   (\xi_0,\ldots,\xi_{t-1})$ and satisfy the
feasibility constraints \eqref{soc-inf-b}.

We assume that a decision maker does not have access to  the distribution of $\xi$ conditional on the context $\eta$, but can directly observe data $\{(\xi_i,\eta_i), i=1,\ldots,N\}$. This assumption arises naturally in many practical applications. For example, in inventory control, customer demand is influenced by seasonal factors and can be directly observed from sales data; in portfolio optimization, asset prices depend on underlying economic factors and are readily observable; and in queueing or service systems, arrival rates and service times may vary with time-related factors and are often directly measurable.

\paragraph{Markovian contextual dynamics and policy simplification.}
We assume that $\{\eta_t\}_{t\ge 1}$ is a time-homogeneous Markov chain on a finite
state space $\HH$
with \emph{known} transition probabilities
\begin{equation}\label{mchain}
 \varpi_{h,h'}\coloneqq  \bbp(\eta_{t+1}=h'| \eta_t=h),\qquad h,h'\in \HH.
\end{equation}
Moreover, conditional on $\eta_t$, the random vector $\xi_t$ has \emph{unknown} density
$q(\cdot| \eta_t)$ and is conditionally independent of the past.
In addition, $\eta_{t+1}$ conditional on $\eta_t$ is independent of $\xi_{[t]}$.
Consequently, the augmented state $z_t\coloneqq  (x_t,\eta_t)$ defines a controlled
Markov process, and \eqref{soc-inf} can be viewed as an infinite-horizon MDP
with state $z_t$ and action $u_t$.

In general, an admissible policy may depend on the entire observed history,
that is, $u_t=\pi_t(x_t,\eta_{[t]},\xi_{[t-1]})$.
Under the above Markovian assumptions, $(x_t,\eta_t)$ is a sufficient state,
and hence there is no loss of optimality in restricting attention to policies
of the form $u_t=\pi_t(x_t,\eta_t,\xi_{[t-1]})$.
Moreover, since $\xi_t$ is conditionally independent of the past given $\eta_t$,
one can further restrict to Markov policies $u_t=\pi_t(x_t,\eta_t)$, which is
the policy class used throughout the paper.
We illustrate the problem formulation in  \Cref{fig:data-process-illustration}.

The restriction to Markovian policies can also be justfied by reformulating the contextual SOC model as a Markov Decision 
Process (MDP)
with state $(x,\eta)$, where the corresponding transition kernel is induced by
the state dynamics, the Markov transition law of $\eta$, and the
conditional distribution $\xi \sim q(\cdot | \eta)$. Hence, by
standard MDP theory, optimization over admissible history-dependent
policies can be restricted, without loss of generality, to Markov
policies.

Restricting to the Markov policies, 
the Bellman equation of \eqref{soc-inf} is  
\begin{equation}\label{bell-inf}
V(x,\eta)=\inf_{u\in\U}\;
\bbe_{\xi\sim q(\cdot| \eta)}
\Big[
c(x,u,\xi)+\gamma\,
\bbe_{\eta'|\eta}
\big[V(F(x,u,\xi),\eta')\big]
\Big],
\end{equation}
where $\eta'$ denotes the next context state with distribution
$\bbp(\eta'=\bar\eta| \eta)=\varpi_{\eta,\bar\eta}$
and the conditional expectation
\begin{equation}\label{exp}
\bbe_{\eta'|\eta}\big[V(F(x,u,\xi),\eta')\big]=
\sum_{\eta'\in \HH} \varpi_{\eta, \eta'}V(F(x,u,\xi),\eta'). 
\end{equation}
Again by the general theory of MDPs equation \eqref{bell-inf} has unique solution $V^*(x,\eta)$. 
Under the specified assumptions, the associated Bellman operator is a
$\gamma$-contraction under the uniform norm on bounded functions on
$\mathcal{X}\times\mathcal{H}$, and hence \eqref{bell-inf} admits a unique
bounded solution $V^*$, which is the optimal value function of
\eqref{soc-inf}.

\begin{figure}[t]
\centering
\begin{tikzpicture}[scale=0.8, transform shape,
    >={Latex[length=2.5mm]},
    node distance=1.5cm,
    semithick,
    font=\small,
    var_node/.style={circle, draw=black, thick, minimum size=0.7cm, inner sep=1pt},
    context/.style={var_node, draw=myred, fill=myred!10},
    control/.style={var_node, draw=mypurple, fill=mypurple!10, rectangle, rounded corners},
    noise/.style={var_node, draw=mygray, fill=mygray!10, dashed},
    state/.style={var_node, draw=myblue, fill=myblue!10},
    link/.style={->, thick, gray!80},
    markov_link/.style={->, thick, myred},
    dynamics_link/.style={->, thick, myblue},
    policy_link/.style={->, thick, mypurple!80!black}
]

    \def\yEta{2.4}
    \def\yXi{1.6}
    \def\yU{0.8}
    \def\yX{0.0}

    \def\xZero{0.4}       
    
    \def\xOneStart{2.0}   
    \def\xOneDec{3.0}     
    \def\xOneEnd{5.4}     
    
    \def\xTwoStart{7.0}   
    \def\xTwoDec{8.0}     
    \def\xTwoEnd{10.4}    
    
    \def\xThreeStart{12.0} 
    \def\xThreeDec{13.0}  
    \def\xThreeEnd{15.4}  
    
    \def\xInf{16.8}


    \node[noise] (xi0) at (\xZero, \yXi) {$\xi_0$};

    \node[context] (eta1) at (\xOneStart, \yEta) {$\eta_1$};
    \node[state] (x1) at (\xOneStart, \yX) {$x_1$};
    \node[control] (u1) at (\xOneDec, \yU) {$u_1$};
    \node[below=0.1cm of u1, xshift=0.5cm, font=\footnotesize, text=mypurple] {$u_1=\pi(x_1,\eta_1)$};
    \node[noise] (xi1) at (\xOneEnd, \yXi) {$\xi_1$};

    \node[context] (eta2) at (\xTwoStart, \yEta) {$\eta_2$};
    \node[state] (x2) at (\xTwoStart, \yX) {$x_2$};
    \node[below=0.1cm of x2, xshift=0.5cm, font=\footnotesize, text=myblue, align=center] {$x_{2}=F(x_1,u_1,\xi_1)$};
    \node[control] (u2) at (\xTwoDec, \yU) {$u_2$};
    \node[below=0.1cm of u2, xshift=0.5cm, font=\footnotesize, text=mypurple] {$u_2=\pi(x_2,\eta_2)$};
    \node[noise] (xi2) at (\xTwoEnd, \yXi) {$\xi_2$};

    \node[context] (eta3) at (\xThreeStart, \yEta) {$\eta_3$};
    \node[state] (x3) at (\xThreeStart, \yX) {$x_3$};
    \node[below=0.1cm of x3, xshift=0.5cm, font=\footnotesize, text=myblue, align=center] {$x_{3}=F(x_2,u_2,\xi_2)$};
    \node[control] (u3) at (\xThreeDec, \yU) {$u_3$};
    \node[below=0.1cm of u3, xshift=0.5cm, font=\footnotesize, text=mypurple] {$u_3=\pi(x_3,\eta_3)$};
    \node[noise] (xi3) at (\xThreeEnd, \yXi) {$\xi_3$};

    \node (etaInf) at (\xInf, \yEta) {$\dots$};
    \node (xiInf) at (\xInf, \yXi) {$\dots$};
    \node (uInf) at (\xInf, \yU) {$\dots$};
    \node (xInf) at (\xInf, \yX) {$\dots$};


    \draw[markov_link] (eta1) -- node[midway, above, font=\footnotesize, text=myred] {$\varpi_{\eta_1,\eta_2}$} (eta2);
    \draw[markov_link] (eta2) -- node[midway, above, font=\footnotesize, text=myred] {$\varpi_{\eta_2,\eta_3}$} (eta3);
    \draw[markov_link] (eta3) -- node[midway, above, font=\footnotesize, text=myred] {$\varpi_{\eta_3,\dots}$} (etaInf);

    \draw[link, myblue] (x1) -- (u1); \draw[link, myred] (eta1) -- (u1);
    \draw[link, myblue] (x2) -- (u2); \draw[link, myred] (eta2) -- (u2);
    \draw[link, myblue] (x3) -- (u3); \draw[link, myred] (eta3) -- (u3);

    \draw[link, dashed, myred] (eta1) -- node[midway, below left, font=\scriptsize, inner sep=1pt] {$q(\cdot|\eta_1)$} (xi1);
    \draw[link, dashed, myred] (eta2) -- node[midway, below left, font=\scriptsize, inner sep=1pt] {$q(\cdot|\eta_2)$} (xi2);
    \draw[link, dashed, myred] (eta3) -- node[midway, below left, font=\scriptsize, inner sep=1pt] {$q(\cdot|\eta_3)$} (xi3);

    \draw[dynamics_link] (x1) to[out=-20, in=200] (x2);
    \draw[link, mypurple] (u1.east) -- (x2);
    \draw[link, mygray] (xi1) -- (x2);

    \draw[dynamics_link] (x2) to[out=-20, in=200] (x3);
    \draw[link, mypurple] (u2.east) -- (x3);
    \draw[link, mygray] (xi2) -- (x3);

    \draw[dynamics_link] (x3) to[out=-20, in=200] (xInf);
    \draw[link, mypurple] (u3.east) -- (xInf);
    \draw[link, mygray] (xi3) -- (xInf);


    \draw[dashed, gray!30] (1.0, \yEta+0.4) -- (1.0, \yX-0.9);
    \draw[dashed, gray!30] (6.0, \yEta+0.4) -- (6.0, \yX-0.9);
    \draw[dashed, gray!30] (11.0, \yEta+0.4) -- (11.0, \yX-0.9);
    \draw[dashed, gray!30] (16.0, \yEta+0.4) -- (16.0, \yX-0.9);

    \node[anchor=east, font=\footnotesize, myred] at (-0.1, \yEta) {Context};
    \node[anchor=east, font=\footnotesize, mygray] at (-0.1, \yXi) {Randomness};
    \node[anchor=east, font=\footnotesize, mypurple] at (-0.1, \yU) {Decision};
    \node[anchor=east, font=\footnotesize, myblue] at (-0.1, \yX) {State};

    \node[font=\bfseries] at (\xZero, \yX-1.3) {$t=0$};
    \node[font=\bfseries] at (\xOneStart+1.5, \yX-1.3) {$t=1$};
    \node[font=\bfseries] at (\xTwoStart+1.5, \yX-1.3) {$t=2$};
    \node[font=\bfseries] at (\xThreeStart+1.5, \yX-1.3) {$t=3$};
    \node[font=\bfseries] at (\xInf, \yX-1.3) {$\ldots$};

\end{tikzpicture}
\caption{Timeline of the data (context and randomness) process, 
system dynamics,  and control processes. The context $\eta_t$ evolves according to 
the known transition probability $\varpi_{\eta_t, \eta_{t+1}}$, and generates the randomness $\xi_t$ via the unknown
density
$q(\cdot|\eta_t)$. The action $u_t = \pi(x_t, \eta_t)$ is chosen based on state and context, driving the system dynamics $x_{t+1}=F(x_t, u_t, \xi_t)$.}
\label{fig:data-process-illustration}
\end{figure}

\begin{example}
Consider the stationary inventory model (e.g., \cite{zipkin})
\begin{equation}\label{inv-1}
\min\limits_{u_t\ge 0}  \bbe \big [ \tsum_{t=1}^{\infty}
\gamma^{t-1}  \big(cu_t+\psi(x_t+u_t,D_t)\big)
\big] \quad {\rm s.t.} \quad x_{t+1}=x_t+u_t-D_t,
\end{equation}
where
$
\psi(x,d)\coloneqq b[d-x]_+ + h[x-d]_+
$,
and 
the cost parameters satisfy $b>c\ge 0$, $h\ge 0$, and
$c+h>0$.
Moreover, conditional on $\eta_t=\eta$, demand satisfies
$D_t\sim q(\cdot|\eta)$, is nonnegative, has finite first moment, 
and is conditionally independent of $\eta_{t+1}$.

In the present case 
the Bellman equation  \eqref{bell-inf} takes the form 
\begin{equation*}
V(x,\eta)=\inf_{u\ge 0}\;
\bbe_{D\sim q(\cdot| \eta)}
\Big[
cu+\psi(x+u,D)+\gamma\,
\bbe_{\eta'|\eta}
\big[V(x+u-D,\eta')\big]
\Big].
\end{equation*}
By the change of variables $y=x+u$ we can write this
equation as
\begin{equation}\label{inv-3}
V(x,\eta)=-cx +\inf_{y\ge x}\;
\bbe_{D\sim q(\cdot| \eta)}
\Big[
cy+\psi(y,D)+\gamma\,
\bbe_{\eta'|\eta}
\big[V(y-D,\eta')\big]
\Big].
\end{equation}

Consider the solution  $V^*(x,\eta)$ of the Bellman equation.   
Note that for every $\eta\in \HH$, the value function $V(\cdot,\eta)$ is convex.
Therefore the minimizer $\bar{y}_\eta$ in the right hand side of \eqref{inv-3} is given by $\bar{y}_\eta=\max\{y^*_\eta,x\}$, where $y^*_\eta$ is the unconstrained minimizer
 \begin{equation}\label{inv-5}
y^*_\eta\in  \argmin_{y} \bbe_{D\sim q(\cdot| \eta)}
\Big[
cy+\psi(y,D)+\gamma\,
\bbe_{\eta'|\eta}
\big[V^*(y-D,\eta')\big]
\Big]. 
\end{equation}
Therefore the optimal policy is the  context-dependent
base-stock policy  
\begin{equation*}
\pi^*(x,\eta)=\max\{y^*_\eta,x\}-x=[y_\eta^*-x]_+. 
\end{equation*}
\end{example}

\paragraph{Bayesian reformulation.}
Since the conditional density \(q(\xi| \eta)\) is unknown, the Bellman equation \eqref{bell-inf} is not available and needs to be estimated. To this end, we take a Bayesian approach. Suppose that the conditional density \(q(\xi| \eta)\) is modeled by a parametric density family \(\{f(\xi| \eta,\theta):\theta\in\Theta\}\) with unknown
parameter \(\theta\), where \(\Theta\subset\mathbb{R}^{s}\).
Let \(p(\theta)\) be a prior density on \(\Theta\).

We observe data $\eta_1,\ldots,\eta_N\in\mathcal{H}$ and
$\xi_1,\ldots,\xi_N$, where $\xi_i$ is generated according to the conditional density $q(\cdot| \eta_i)$ and is conditionally independent of $\xi_{[i-1]}$ given $\{\eta_i \}_{i=1}^N$. 
Moreover, $\eta_1$ is deterministic. 
Given data $(\xi_i,\eta_i)$, $i=1,\ldots,N$, the posterior density is
\begin{equation}\label{post-bayes-inf}
\mathsf{p}_N(\theta)=
\frac{\sff_{N}(\theta)p(\theta)}
{\int_{\Theta}\sff_{N}(\vartheta)p(\vartheta)\,\mathrm{d}\vartheta},
\qquad
\text{where}
\qquad 
\sff_{N}(\theta)\coloneqq  
f(\xi_{1}| \eta_{1},\theta)
\prod_{i=2}^{N} \big[ f(\xi_i| \eta_i,\theta)\
\varpi_{\eta_{i-1},\eta_i}
\big].
\end{equation}
For a bounded random variable $Y$,
its posterior predictive conditional expectation is defined by
\begin{equation}\label{bell-3}
\bbe_{\mathsf{p}_N| \eta}[Y]
\coloneqq  
\bbe_{\theta\sim \mathsf{p}_N}\Big[\bbe_{\xi\sim f(\cdot| \eta,\theta)}[Y]\Big]
=
\int_{\Theta}\int_\Xi Y(\xi)\,f(\xi| \eta,\theta)\,\mathsf{p}_N(\theta)\,\mathrm{d}\xi\,\mathrm{d}\theta.
\end{equation}
The corresponding Bayesian value function $V_N^*$ on $\mathcal{X} \times \mathcal{H}$
 is defined as the optimal
value 
that satisfies the Bayesian Bellman equation
\begin{equation}\label{bell-bayes-inf}
V_N^*(x,\eta)=\inf_{u\in\U}\;
\bbe_{\mathsf{p}_N| \eta}\Big[
c(x,u,\xi)+\gamma\,
\bbe_{\eta'|\eta}
\big[V_N^*(F(x,u,\xi),\eta')\big]
\Big].
\end{equation}
Compared with the Bellman equation in \eqref{bell-inf}, the
unknown context-conditional disturbance law is replaced by its posterior
predictive counterpart.
By the same contraction argument as above, for each $N\in\mathbb{N}$,
\eqref{bell-bayes-inf} admits a unique bounded solution $V_N^*$.
If it exists, a
Bayesian optimal (stationary Markov) policy is given by a  measurable
 selector  
\begin{equation}\label{optsol}
  \pi_N^*(x,\eta)\in\argmin_{u\in\U}  
\bbe_{\mathsf{p}_N| \eta}\Big[
c(x,u,\xi)+\gamma\,
\bbe_{\eta'|\eta}
\big[V_N^*(F(x,u,\xi),\eta')\big]
\Big].
\end{equation}

\paragraph{Sequential estimate-then-control process.}
After observing the exogenous data stream up to time \(N\), namely
\(\{(\xi_i,\eta_i)\}_{i=1}^N\), the posterior \(\mathsf{p}_N\) is
constructed according to \eqref{post-bayes-inf}. This posterior defines
the posterior predictive expectation in \eqref{bell-3} and hence the
Bayesian Bellman equation \eqref{bell-bayes-inf}, whose solution gives
\(V_N^*\) and the policy \(\pi_N^*\) in \eqref{optsol}. With
\(\mathsf{p}_N\) fixed, \eqref{bell-bayes-inf} is an infinite-horizon
control problem over stages \(t=1,2,\ldots\). When a new observation
\((\xi_{N+1},\eta_{N+1})\) arrives, the same construction is
repeated with \(\mathsf{p}_{N+1}\).
We illustrate this  scheme in \Cref{fig:bayesian}.

\begin{figure}[t]
\centering
\begin{tikzpicture}[scale=0.8, transform shape,
    font=\small,
    >=Latex,
    node distance=0.65cm,
    define color/.code={
        \definecolor{myblue}{RGB}{70, 130, 180}    
        \definecolor{myred}{RGB}{205, 92, 92}      
        \definecolor{mygray}{RGB}{105, 105, 105}   
        \definecolor{mypurple}{RGB}{153, 102, 204} 
        \definecolor{myorange}{RGB}{218, 135, 56}  
    },
    define color,
    database/.style={
        cylinder,
        shape border rotate=90,
        aspect=0.25,
        draw=myred,
        thick,
        fill=myred!10,
        minimum width=2.2cm,
        minimum height=1.8cm,
        align=center,
    },
    posterior/.style={
        rectangle,
        rounded corners=2mm,
        draw=mygray,
        thick,
        fill=mygray!10,
        minimum width=2.2cm,
        minimum height=1.5cm,
        align=center,
    },
    value/.style={
        rectangle,
        rounded corners=2mm,
        draw=myblue,
        thick,
        fill=myblue!10,
        minimum width=2.2cm,
        minimum height=1.5cm,
        align=center,
    },
    policy/.style={
        rectangle,
        rounded corners=2mm,
        draw=mypurple,
        thick,
        fill=mypurple!10,
        minimum width=2.2cm,
        minimum height=1.5cm,
        align=center,
    },
    implementation/.style={
        rectangle,
        rounded corners=2mm,
        draw=myorange,
        thick,
        fill=myorange!10,
        minimum width=2.2cm,
        minimum height=1.5cm,
        align=center,
    },
    flow/.style={
        ->,
        line width=1.2pt,
        draw=gray!80
    }
]

    \node[database] (history) {
        Data \\
        $\{(\xi_i, \eta_i)\}_{i=1}^N$
    };

    \node[posterior, right=of history] (posterior_node) {
        Posterior \\
        $\mathsf{p}_N$ in \eqref{post-bayes-inf}
    };

    \node[value, right=of posterior_node] (bellman) {
        Bayesian value \\
        function \\
        $V_N^*$ in \eqref{bell-bayes-inf}
    };

    \node[policy, right=of bellman] (policy_node) {
        Bayesian \\
        policy \\
        $\pi_N^*$ in \eqref{optsol}
    };

    \node[implementation, right=of policy_node] (implement) {
        Implement  \\
policy $\pi_N^*$ for \\
        one stage
    };

    \node[database, right=of implement] (newdata) {
        Observe \\
        $(\xi_{N+1},\eta_{N+1})$
    };

    \draw[flow] (history) -- (posterior_node);
    \draw[flow] (posterior_node) -- (bellman);
    \draw[flow] (bellman) -- (policy_node);
    \draw[flow] (policy_node) -- (implement);
    \draw[flow] (implement) -- (newdata);

    \draw[flow]
        (newdata.south)
        -- ++(0,-0.8)
        -| node[pos=0.25, below, font=\footnotesize]
            {update data and repeat}
        (history.south);

\end{tikzpicture}
\caption{Schematic of the Bayesian
sequential estimate-then-control pipeline after observing
exogenous data up to time \(N\). The accumulated historical data
\(\{(\xi_i,\eta_i)\}_{i=1}^N\) are used to
construct the posterior \(\mathsf{p}_N\), which defines the
posterior predictive expectation
used to solve for the Bayesian value function \(V_N^*\) and the
corresponding optimal policy \(\pi_N^*\). This policy is implemented for
one period; after observing \((\xi_{N+1},\eta_{N+1})\), the data set is
updated and the procedure is repeated.}
\label{fig:bayesian}
\end{figure}

\begin{remark}
The Bayesian Bellman equation \eqref{bell-bayes-inf} can be viewed as a Bayesian average approximation of the original Bellman equation \eqref{bell-inf}. It is also possible to replace the expectation $\bbe_{\theta\sim\mathsf{p}_N}$ with a risk measure with respect to the posterior, leading to a Bayesian risk formulation inspired by \cite{wu2018,Zhou2022NIPS}. The risk measure (including expectation) represents the decision maker's risk attitude towards the epistemic uncertainty induced by the unknown context-conditional distribution of the randomness.
\end{remark}

\begin{remark}
It is important to distinguish the role of the randomness $\xi$ in the learning versus the control phase.
The posterior density $\mathsf{p}_N$ is constructed using the historical dataset $\{(\xi_i, \eta_i)\}_{i=1}^N$.
For each fixed \(N\), this posterior is treated as fixed when
solving the Bayesian Bellman equation \eqref{bell-bayes-inf}. Thus
\(N\) indexes the amount of historical exogenous data available for
learning, while \(t\) indexes the stages \(t=1,2,\ldots\) of the
infinite-horizon control problem. The paper studies the sequence of
data-dependent Bellman equations, value functions, and policies as
\(N\to\infty\); it does not study an adaptive Bayesian control problem
in which the posterior is part of the controlled state.
In contrast, the variable $\xi$ appearing in the expectations 
\eqref{bell-3} and \eqref{bell-bayes-inf} represents a generic future uncertain randomness.
While the control must be chosen before this future $\xi$ is realized (see \Cref{fig:data-process-illustration}), the expectation $\bbe_{\mathsf{p}_N| \eta}$ in  \eqref{bell-3} averages over the posterior predictive distribution given the historical data.
\end{remark}

\section{Consistency}

We analyze Bayesian learning and control in the Markovian  setting.
We introduce notation and probability laws necessary for our Bayesian analysis.
Subsequently, we prove posterior consistency with Markov samples and the
resulting uniform consistency of the Bayesian value functions.

\paragraph{Notation and terminology.}
For a set $\mathcal{Y}$, let $\mathbb{B}(\mathcal{Y})$ denote the space of
bounded real-valued functions on $\mathcal{Y}$. For $V\in\mathbb{B}(\mathcal{Y})$,
we write $\|V\|_{\infty}\coloneqq   \sup_{y\in\mathcal{Y}}|V(y)|$ for the uniform
norm. For vectors $v\in\mathbb{R}^{s}$, we write $\|v\|_2$ for the Euclidean norm.
For a set $E\subset \mathbb{R}^{s}$, we denote its interior by $\mathrm{int}(E)$.
For $y\in\mathbb{R}^{s}$ and $\varepsilon>0$, define the open ball
$B_{\varepsilon}(y)\coloneqq   \{z\in\mathbb{R}^{s}:\|z-y\|_2<\varepsilon\}$ and its
closure $\overline{B}_{\varepsilon}(y)\coloneqq  \{z\in\mathbb{R}^{s}:\|z-y\|_2\le
\varepsilon\}$. By $|\Sigma|$ we denote the determinant of matrix $\Sigma$. 
By $\dist(\theta,\Theta)\coloneqq \inf_{\theta'\in \Theta}\|\theta-\theta'\|_2$  we denote the distance from  $\theta \in \mathbb{R}^{s}$ to the set $\Theta$. 
Let $Q$ be a probability measure. 
We write $\mathbb{E}_{Q}$ for the expectation with respect to the
law $Q$.
In asymptotic statements,
``$\xrightarrow{Q}$'' denotes convergence in probability under $Q$, and
``$\overset{Q}{\rightsquigarrow}$'' denotes convergence in distribution under
$Q$.

In the Bayesian consistency and Bernstein--von Mises-type asymptotic analysis, we distinguish
between the true data-generating law, the parametric laws indexed by $\theta$, 
the corresponding stationary distribution of $\{(\xi_t, \eta_t)\}_{t \geq 1}$, 
and the posterior. We introduce these laws now.

\begin{itemize}
\item For $\theta\in\Theta$, let $\mathbb{P}_\theta$ denote the probability law of
the process $\{(\xi_t,\eta_t)\}_{t\ge 1}$ when $\xi_t$ has density
$f(\cdot| \eta_t,\theta)$ and $\{\eta_t\}_{t\ge 1}$ evolves according to
\eqref{mchain}.
\item We denote by $\mathbb{P}^*$ the distribution of the true data-generating process
$(\xi_t,\eta_t)$, $t=1,2, \ldots$. Please note that $\mathbb{P}^*$  
is identical to $\mathbb{P}_{\theta^*}$, 
if the model is correctly specified  at $\theta^*$ as defined in \Cref{def-cor}  below.
\item When  the Markov chain \eqref{mchain} has a stationary distribution
$\nu_\eta$, 
 we denote the corresponding
stationary one-period law by $P^*$, that is,
\begin{equation}
\label{eq:Pstar}
\mathrm{d}P^*(\xi,\eta)
\coloneqq
\nu_\eta(\eta)q(\xi|\eta)\,\mathrm{d}\xi .
\end{equation}

\item Let $\mathsf{P}_N$ be the probability measure induced by the posterior density $\mathsf{p}_N$, that is, it is the posterior measure $\mathsf{P}_N(A) \coloneqq  \int_{A} \mathsf{p}_N(\theta) \mathrm{d}\theta$
defined for measurable $A \subset \Theta$.
\end{itemize}

Throughout, $\mathbb{P}^*$ and $\mathbb{P}_\theta$ denote path-space
laws for the full process $\{(\xi_t,\eta_t)\}_{t\ge 1}$, while $P^*$
denotes the stationary one-period marginal law of a generic pair
$(\xi,\eta)$ under the true data-generating process. Consequently,
$\mathbb{E}_{P^*}$ is used for population expectations under the
stationary marginal law, whereas almost-sure, in-probability, and
weak-convergence statements are taken with respect to $\mathbb{P}^*$.

\subsection{Consistency of Bayesian posterior with Markov samples}
\label{sec:consistency}
It is well known that the Bayesian posterior is consistent with independent and identically distributed (i.i.d.) samples (see, e.g., \cite{vaart}). In our setting, the data samples $\{(\xi_i,\eta_i)\}_{i \geq 1}$ are not i.i.d.\ because $\{\eta_i\}_{i \geq 1}$ is generated from a Markov chain. Bayesian consistency still holds with  Markovian samples (see, e.g., \cite{GT06, GV07, KA25}). The results in these papers are usually proven by assuming entropy or existence of certain test statistics, for general stochastic processes such as $\alpha$-mixing Markov chains. Instead, we show Bayesian consistency under simpler assumptions but only for ergodic Markov chains. Our proof is built on the proof of Bayesian consistency for i.i.d.\ data in \cite{SZY2023}.

Suppose that the Markov chain \eqref{mchain} admits a stationary distribution
$\nu_\eta$. Then, as noted above, the joint process $\{(\xi_t,\eta_t)\}_{t\ge 1}$
is a Markov chain with stationary distribution
$P^*$.
For $\theta\in\Theta$, define the population log-likelihood
\begin{equation}\label{eq:psi}
\lambda(\theta)\coloneqq \bbe_{P^*}\left[\log f(\xi| \eta,\theta)\right]
=\sum_{\eta\in\HH}\nu_\eta(\eta) \int_\Xi \log f(\xi| \eta,\theta)\, q(\xi| \eta)\,\mathrm{d}\xi,
\end{equation}
where $P^*$ is defined in \eqref{eq:Pstar}.
Then, up to an additive constant independent of $\theta$,
$-\lambda(\theta)$ is the conditional Kullback--Leibler (KL) divergence from $q(\cdot| \eta)$ to $f(\cdot| \eta,\theta)$ averaged under the distribution $\nu_{\eta}$.
Let
\begin{equation}\label{eq:theta-star-set}
\Theta^*\coloneqq \mathrm{argmax}_{\theta\in\Theta}\, \lambda(\theta)
=\mathrm{argmin}_{\theta\in\Theta} \sum_{\eta\in\HH}\nu_\eta(\eta)
\int_\Xi q(\xi| \eta)\log\!\left(\frac{q(\xi| \eta)}{f(\xi| \eta,\theta)}\right)\!\mathrm{d}\xi.
\end{equation}
We define the empirical log-likelihood as 
\begin{equation}
\label{eq:phi_N}
\lambda_N(\theta)\coloneqq \frac{1}{N}\sum_{i=1}^N \log f(\xi_i| \eta_i,\theta).
\end{equation}
We now make the following assumptions.

\begin{assumption}\label{assump:consistency}
  {\rm (i)} The parameter set   $\Theta\subset\mathbb{R}^{s}$ is convex and compact with nonempty interior.
  {\rm (ii)}  The prior density $p$ is bounded and bounded away from $0$ on $\Theta$,
that is,  there exist constants $c_1\ge c_2>0$ such that
  $c_1\ge p(\theta)\ge c_2$ for all $\theta\in\Theta$.
  {\rm (iii)}  For every $(\xi,\eta)\in \Xi\times \HH$, $f(\xi |\eta,\theta)>0$ for all $\theta\in\Theta$. 
   {\rm (iv)}
 For every $(\xi,\eta)\in \Xi\times \HH$,
    $\theta\mapsto f(\xi|\eta,\theta)$ is continuous on $\Theta$.
  {\rm (v)} The Markov chain $\{\eta_t\}_{t \geq 1}$ is irreducible and aperiodic.
  {\rm (vi)} $\log f(\xi|\eta,\theta)$ is dominated by an integrable function with respect to $P^*$. 
\end{assumption}

Now, we demonstrate that the empirical log-likelihood satisfies a uniform law
of large numbers (LLN) under the true Markov data-generating process.

\begin{lemma}\label{lem:ULLN}
    Under \Cref{assump:consistency}\textup{(i,iii,iv,v,vi)} the following uniform LLN holds:
  \[
  \lim_{N\to\infty}\sup_{\theta\in\Theta}\left|
  \lambda_N(\theta)-\lambda(\theta)\right|=0, \;\;
\mathbb{P}^*\text{-almost surely.}
  \]
\end{lemma}

\begin{proof}{Proof}
    \Cref{assump:consistency}(v) implies $\{\eta_t\}_{t \geq 1}$ is a (geometrically) ergodic Markov chain, and moreover, $\{(\xi_t,\eta_t)\}_{t \geq 1}$ is also ergodic. Hence, for each fixed $\theta \in \Theta$, regardless of the initial condition $(\xi_1,\eta_1)$, $\lim_{N\rightarrow\infty} |\lambda_N(\theta) -\lambda(\theta)| =0$, $\mathbb{P}^*$-almost surely. With the dominated integrability of $\log f(\xi|\eta,\theta)$ and the compactness of $\Theta$, the LLN holds uniformly over $\Theta$.
\end{proof}

With this uniform LLN in hand, we next show that the posterior density
\(\mathsf{p}_N\) decays exponentially outside any neighborhood of \(\Theta^*\).

\begin{lemma}[Exponential decay away from $\Theta^*$]\label{lem:exp-decay}
For $\theta^*\in \Theta^*$, define $V_\e\coloneqq \{\theta\in \Theta:\lambda(\theta^*)-\lambda(\theta)\ge \e\}$ and $U_\e\coloneqq \Theta\setminus V_\e$. Suppose that \Cref{assump:consistency} holds.
Then for any $\varepsilon>0$ and any $0<\beta<\alpha<\varepsilon$, with $\mathbb{P}^*$-probability $1$ for all sufficiently large $N$,
\begin{equation}\label{eq:exp-decay}
\sup_{\theta\in V_\varepsilon}\mathsf{p}_N(\theta)\le \kappa(\beta)^{-1}(c_1/c_2)^2\exp\{-N(\alpha-\beta)\},
\end{equation}
where $\kappa(\beta)\coloneqq \int_{U_{\beta/2}} \mathrm{d}\theta>0$.
\end{lemma}

\begin{proof}{Proof}
   It follows by the same argument as Lemma~3.1 in \cite{SZY2023}, with the i.i.d.\ uniform LLN replaced by the uniform LLN in \Cref{lem:ULLN}.
\end{proof}

We recall that $\mathsf{P}_N$ is the probability measure induced by the
posterior density $\mathsf{p}_N$. For standard posterior-consistency
notions in the i.i.d.\ setting, we refer the reader to Section~10.4 in \cite{vaart}.

\begin{theorem}[Posterior consistency]\label{thm:consistency}
Suppose that \Cref{assump:consistency} holds.
Let $\theta_N$ be a random variable with the  posterior density $\mathsf{p}_N(\cdot)$ defined in \eqref{post-bayes-inf}.
Then,  $\mathbb{P}^*$-almost surely,
\begin{equation}\label{conprob}
 \dist(\theta_N,\Theta^*)\xrightarrow{\mathsf{P}_N} 0,\quad  \text{as} \quad  N\to\infty.
\end{equation}
In particular, if $\Theta^*=\{\theta^*\}$ is the singleton, then
$\mathbb{P}^*$-almost surely, $\theta_N\xrightarrow{\mathsf{P}_N}\theta^*$.
\end{theorem}

\begin{proof}{Proof}
Fix $\varepsilon>0$.
By definition, $\Theta\setminus U_\varepsilon=V_\varepsilon$.
Therefore
\[
\mathsf{P}_N\big(\theta_N\in V_\varepsilon\big)
=\int_{V_\varepsilon}\mathsf{p}_N(\theta)\,\mathrm{d}\theta
\le \left(\int_{\Theta} \mathrm{d}\theta\right)\sup_{\theta\in V_\varepsilon}\mathsf{p}_N(\theta).
\]
By \Cref{lem:exp-decay}, the right-hand side converges to $0$ $\mathbb{P}^*$-almost surely, so
$\mathsf{P}_N(\theta_N\in U_\varepsilon)\to 1$ $\mathbb{P}^*$-almost surely. The sets $V_\e$ and $U_\e$ do not depend on a particular choice of $\theta^*\in \Theta^*$, and the sets $\{U_\varepsilon, \varepsilon > 0\}$ shrink to $\Theta^*$ as $\varepsilon\downarrow 0$. This implies the claimed convergence of $\dist(\theta_N,\Theta^*)$ in probability.
 \end{proof}

\Cref{thm:consistency} involves two sources of randomness. The ``outer'' randomness
comes from the data sequence $\{(\xi_i,\eta_i)\}_{i\ge 1}$ generated under the
true law $\mathbb{P}^*$, which makes the posterior measure $\mathsf{P}_N$
(equivalently, the density $\mathsf{p}_N$) a random object. Conditional on the
observed data, $\theta_N$ is an ``inner'' random variable drawn from the posterior,
that is, $\theta_N\sim \mathsf{P}_N$.

\subsection{Consistency of value functions}

We now turn to consistency of the Bayesian value functions under correct
specification and identifiability of the randomness model.

\begin{definition}
\label{def-cor}
It is said that the model is {\em correctly specified} if there exists $\theta^*\in \Theta$ such that  $q(\cdot|\eta)=f(\cdot|\eta, \theta^*)$ for all $\eta\in \HH$.
It is said that the model is {\em  identifiable} (at $\theta^*$)  if 
$f(\cdot|\eta, \theta^*)=f(\cdot|\eta, \theta')$,
for all $\eta \in \HH$ with
$\theta'\in \Theta$, implies $\theta'=\theta^*$. 
\end{definition}

From now on, we assume that the model is   correctly specified and  identifiable  at $\theta^*$.
Recall the definition of $\lambda(\theta)$ and      $\Theta^*$ in \eqref{eq:psi} and \eqref{eq:theta-star-set}, respectively.
For each $\eta\in\HH$, the quantity
$
\int_\Xi q(\xi| \eta)
\log\big(
\frac{q(\xi| \eta)}{f(\xi| \eta,\theta)}
\big)
\,\mathrm{d}\xi
$
is the KL divergence from $q(\cdot| \eta)$ to
$f(\cdot| \eta,\theta)$, and hence is nonnegative. Since the
Markov chain on the finite state space $\HH$ is irreducible, its
stationary distribution satisfies $\nu_\eta(\eta)>0$ for every
$\eta\in\HH$. Therefore, if the model is correctly specified at
$\theta^*$, then the weighted KL divergence in
\eqref{eq:theta-star-set} is minimized at $\theta^*$, and hence
\[
\theta^*
\in
\argmax_{\theta\in\Theta}
\bbe_{P^*}[\log f(\xi | \eta,\theta)],
\]
where $P^*$ is defined in \eqref{eq:Pstar}. If, in addition, the
model is identifiable at $\theta^*$, then this maximizer is unique.
By \Cref{thm:consistency} this implies that the Bayesian posterior $\mathsf{p}_N$ almost surely converges to a $\theta^*$ that maximizes $\bbe_{P^*}[\log f(\xi|\eta,\theta)]$, the population log-likelihood of the context-randomness pair $(\xi,\eta)$, as the data size $N$ increases to infinity. That is, 
under \Cref{assump:consistency}  $\mathbb{P}^*$-almost surely, $\theta_N$ converges in probability to $\theta^*$ with respect to the posterior 
measure $\mathsf{P}_N$.

\begin{assumption}
\label{assume:infinite-horizon-1}
The model is correctly specified  and the true parameter $\theta^* \in \mathrm{int}(\Theta)$ 
is  identifiable.
\end{assumption}

We demonstrate the asymptotic consistency of the 
Bayesian value function $V_N^*$ toward
the true value function $V^*$.

\begin{proposition}
\label{prop:consistency-value-functions}
If \Cref{assump:consistency,assume:infinite-horizon-1} hold, 
then $V_N^*$ converges to $V^*$ uniformly on $\mathcal{X} \times \mathcal{H}$ 
as $N \to \infty$, $\mathbb{P}^*$-almost surely, 
that is  
$\|V^*_N - V^*\|_{\infty} \to 0$
as $N \to \infty$, $\mathbb{P}^*$-almost surely.
\end{proposition}

We establish \Cref{prop:consistency-value-functions} using the following lemma,
which shows that, for each context state $\eta\in\mathcal{H}$, the posterior
predictive density converges to the true conditional density in mean. 
 Recall that the cost function $c \colon \mathcal{X} \times \mathcal{U} \times \Xi \to \mathbb{R}$
is assumed to be bounded. 

\begin{lemma}
\label{lem:predictive-L1}
If \Cref{assump:consistency,assume:infinite-horizon-1} hold, 
then for all $\eta\in\mathcal{H}$, $\mathbb{P}^*$-almost surely
$$
\mathbb{E}_{\theta\sim \mathsf{p}_N}\Big[
\int_{\Xi}
|
f(\xi|\eta,\theta)
-
f(\xi|\eta,\theta^*)
|
\,\mathrm{d}\xi
\Big]
\to 0,\;\;as\;N \to \infty.
$$
\end{lemma}

\begin{proof}{Proof}
Fix $\eta\in\mathcal{H}$.
Let $\varepsilon> 0$.
Standard arguments (see, e.g., eqns.~(22) and (23), and p.~6 in \cite{SZLW2025})
ensure
\begin{equation}
\label{eq:posterior-mean-bound}
\begin{aligned}
\mathbb{E}_{\theta\sim \mathsf{p}_N}\Big[
\int_{\Xi}
|
f(\xi|\eta,\theta)
-
f(\xi|\eta,\theta^*)
|
\,\mathrm{d}\xi
\Big]
& \le 
\sup_{\theta\in \overline{B}_{\varepsilon}(\theta^*)}
\int_{\Xi}\big|f(\xi|\eta,\theta)-f(\xi|\eta,\theta^*)\big|\,\mathrm{d}\xi
\\
& \quad +
2\mathsf{P}_N(\overline{B}_{\varepsilon}(\theta^*)^c).
\end{aligned}
\end{equation}
Since $f(\xi | \eta, \cdot)$ is continuous and $\{f(\xi | \eta, \theta),\theta\in\Theta\}$ are densities, 
Scheff\'e's lemma ensures
$
\int_{\Xi}\big|f(\xi|\eta,\theta)-f(\xi|\eta,\theta^*)\big|\,\mathrm{d}\xi
\to 0
$
as $\theta \to \theta^*$.
\Cref{thm:consistency} ensures $\mathbb{P}^*$-almost surely, 
$\mathsf{P}_N(\overline{B}_{\varepsilon}(\theta^*)^c) \to 0$.
Combined with \eqref{eq:posterior-mean-bound}, we obtain the assertion. 
\end{proof}

\begin{proof}{Proof of \Cref{prop:consistency-value-functions}}
The proof is inspired by those of Propositions~1 and 3 in  \cite{SZLW2025}. 
We define the Bellman operators 
$\mathcal{T}$, $\mathcal{T}_N \colon \mathbb{B}(\mathcal{X} \times \mathcal{H}) \to \mathbb{B}(\mathcal{X} \times \mathcal{H})$
by
\begin{equation}
\label{eq:bellman-operators}
\begin{aligned}
[\mathcal{T} V](x, \eta) &\coloneqq  
\inf_{u\in\U}\;
\bbe_{\xi\sim q(\cdot| \eta)}
\Big[
c(x,u,\xi)+\gamma\,
\bbe_{\eta'|\eta}
\big[V(F(x,u,\xi),\eta')\big]
\Big],
\\
[\mathcal{T}_N V](x, \eta) &\coloneqq  
\inf_{u\in\U}\;
\bbe_{\mathsf{p}_N| \eta}\Big[
c(x,u,\xi)+\gamma\,
\bbe_{\eta'|\eta}
\big[V(F(x,u,\xi),\eta')\big]
\Big].
\end{aligned}
\end{equation}%
Since $\gamma \in (0,1)$, 
the Bellman operators  $\mathcal{T}$ and $\mathcal{T}_N$
are $\gamma$-contractions. This ensures the standard error bound
(see, e.g., eq.~(19) in \cite{SZLW2025} and eq.~(4.2) in \cite{Shapiro2020})
\begin{equation*}
\|V^*_N - V^*\|_{\infty}
\leq 
(1-\gamma)^{-1}
\|\mathcal{T}_NV^* - \mathcal{T}V^*\|_\infty.
\end{equation*}
Since $c$ and $V^*$ are bounded, there exists a constant $C> 0$ such that
for all $(x, \eta) \in \mathcal{X} \times \mathcal{H}$,
\begin{equation*}
\Big|[\mathcal{T}_NV^*](x,\eta)-[\mathcal{T}V^*](x,\eta)\Big |
\leq C\
\mathbb{E}_{\theta\sim \mathsf{p}_N}
\Big[
\int_{\Xi}
|
f(\xi|\eta,\theta)
-
f(\xi|\eta,\theta^*)
|
\,\mathrm{d}\xi
\Big].
\end{equation*}
Combined with  \Cref{lem:predictive-L1} and the fact that $\mathcal{H}$ is finite, 
we obtain the assertion.
\end{proof}

\section{Asymptotics of the contextual optimal value}

We now study the first-order asymptotic behavior of the
data-driven contextual optimal value. Specifically,
for each fixed initial condition 
$(x_1, \eta_1) \in \mathcal{X} \times \mathcal{H}$,
we  derive the
limiting distribution for \(N^{1/2}\big(V_N^*(x_1, \eta_1)-V^*(x_1, \eta_1)\big)\) through a
Bernstein--von Mises-type analysis and thereby quantify how posterior
uncertainty in the learned randomness model propagates to the optimal
value.
Our analysis first records the
needed Markov-chain posterior approximation, then applies it to the
contextual optimal value, and finally treats the finite state-action SOC model

\subsection{A weighted local posterior approximation for Markov-chain data}

Inspired by \cite{Borwanker1971, Gillert1983}, we formulate 
a Bernstein--von Mises 
approximation for Markov chains. Acknowledging the fact that the
precise assumptions required to formulate a Bernstein--von Mises theorem
for Markov chains are technical and rather lengthy, we restrict our attention
to
a Markov-chain Bernstein--von Mises approximation stated as an
assumption, and then prove its posterior delta-method consequence in
\Cref{lem:posterior-delta}. We provide
sufficient technical background to make
this formulation meaningful.
In the next section, we establish Bernstein--von Mises-type limits for the
contextual optimal value, assuming that the results stated in this
section apply.
The Bernstein--von Mises theorems in \cite{Borwanker1971, Gillert1983}
apply only to one-dimensional parameter spaces; accordingly, we
state the multivariate local posterior approximation needed for our
analysis as an assumption.

Let 
$Z=\{Z_t\}_{t\ge 1}$ be a time-homogeneous Markov chain 
with measurable state space $\mathcal{Z}$
with arbitrary initial distribution
having density \(h_1(\cdot;\theta)\).
For each $\theta\in\Theta$, the Markov
chain is specified by transition densities
$h(z'| z, \theta)$ for $z,z'\in\mathcal{Z}$.
Given observations $Z_i$, $i=1, \ldots, N$, and a prior density
$r(\theta)$ on $\Theta$, the posterior density is given by
\begin{equation*}
\mathsf{r}_N(\theta)=
\frac{\mathsf{h}_{N}(\theta)r(\theta)}
{\int_{\Theta}\mathsf{h}_{N}(\vartheta)r(\vartheta)
\,\mathrm{d}\vartheta},
\qquad
\text{where}
\qquad 
\mathsf{h}_{N}(\theta)\coloneqq 
h_1(Z_1;\theta)
\prod_{i=2}^{N} h(Z_i| Z_{i-1}, \theta),
\end{equation*}
where we assume that
$\int_{\Theta}\mathsf{h}_{N}(\vartheta)r(\vartheta)
\,\mathrm{d}\vartheta > 0$.

Let $\theta^*\in\mathrm{int}(\Theta)$ be the true parameter, that is,
$Z_1$ has density \(h_1(\cdot;\theta^*)\), and
$Z_i| Z_{i-1}=z \sim h(\cdot| z, \theta^*)$ for $i=2, 3, \ldots$.
We assume that $\theta\mapsto h(y| z, \theta)$ is differentiable at
$\theta^*$ for all $(y,z) \in \mathcal{Z} \times \mathcal{Z}$ such that $h(y| z, \theta^*)>0$.
We define the Fisher information matrix at $\theta^*$ by
\begin{equation}
\label{eq:fisher-matrix-general-setting}
\mathfrak{I}(\theta^*)
\coloneqq 
\mathbb{E}_{(Z_1,Z_2)\sim \mu_{\theta^*}(\cdot)h(\cdot | \cdot, \theta^*)}
\big[
\nabla_\theta \log h(Z_2| Z_1, \theta^*) \nabla_\theta \log h(Z_2| Z_1, \theta^*)^\top
\big],
\end{equation}
where the expectation is taken over the joint law of $(Z_1,Z_2)$,
and $\mu_{\theta^*}$ denotes a stationary distribution of the Markov chain under the 
true parameter $\theta^*$.

Let $\mathfrak{P}^*$ be the law of 
$\{Z_t\}_{t\geq 1}$ under the true parameter
$\theta^*$.
Let $\hat\theta_N$ be a
maximum likelihood estimator based on $\mathsf h_N$, and let
$\mathsf r_N^*$ denote the posterior density of the local parameter
$\tau=N^{1/2}(\theta-\hat\theta_N)$.

\begin{assumption}[Weighted Bernstein--von Mises approximation]
\label{assump:weighted-bvm}
The true parameter satisfies 
$\theta^* \in \mathrm{int}(\Theta)$, 
the matrix $\mathfrak I(\theta^*)$ is nonsingular, 
$
N^{1/2}(\hat\theta_N-\theta^*)
\overset{\mathfrak P^*}{\rightsquigarrow}
\mathcal N(0,\mathfrak I(\theta^*)^{-1})
$,
and 
\begin{equation}
\label{eq:BvM-Markov-Chain}
\lim_{N\to\infty}
\int_{\mathbb R^s}
(1+\|\tau\|_2^2)
\big|\mathsf r_N^*(\tau)-\phi_{\mathfrak I}(\tau)\big|
\,\mathrm d\tau
=0,
\quad \mathfrak P^*\text{-almost surely},
\end{equation}
where 
$\phi_{\mathfrak I}$ is the density of
$\mathcal N(0,\mathfrak I(\theta^*)^{-1})$.
\end{assumption}

\Cref{assump:weighted-bvm} is the only Markov-chain Bernstein--von
Mises input used below. Weighted Bernstein--von Mises
versions such as \eqref{eq:BvM-Markov-Chain} 
follow under
regularity conditions in classical Bernstein--von Mises theory for
Markov chains; see, for example, \cite{Borwanker1971,Gillert1983,Bickel1969}. The
weight $1+\|\tau\|_2^2$ provides both total-variation convergence of the
local posterior and the moment control needed for posterior expectations.

\begin{lemma}
\label{lem:posterior-delta}
If Assumptions~\textup{\ref{assump:consistency}(i)} and 
\textup{\ref{assump:weighted-bvm}} hold, and 
\(g:\Theta\to\mathbb R\) is bounded and differentiable at
\(\theta^\ast \in \mathrm{int}(\Theta)\), then
\begin{equation}
\label{eq:posterior-mean}
N^{1/2}\big(
\mathbb{E}_{\theta \sim \mathsf{r}_N}[g(\theta)]-g(\theta^*)
\big)
\overset{\mathfrak{P}^*}{\rightsquigarrow}
\mathcal{N}\big(
0,
\nabla g(\theta^*)^\top \mathfrak{I}(\theta^*)^{-1}\nabla g(\theta^*)
\big).
\end{equation}
\end{lemma}

We defer the proof to \Cref{app:posterior-delta}. 
The next lemma records an immediate second-moment
implication; it is used 
later for the finite state-action SOC model.

\begin{lemma}
\label{lem:posterior-second-moment}
If Assumptions~\textup{\ref{assump:consistency}(i)} and 
\textup{\ref{assump:weighted-bvm}} hold, then
$
\mathbb{E}_{\theta\sim\mathsf{r}_N}
[\|\theta-\theta^*\|_2^2]
=
O_{\mathfrak{P}^*}(N^{-1})
$.
\end{lemma}
We defer the proof to \Cref{app:posterior-second-moment}.

\subsection{Asymptotics of the contextual optimal value}
\label{subsect:asymptotics}

This section establishes Bernstein--von Mises-type asymptotics for the
data-driven contextual optimal value. Our asymptotic results are based on
the posterior Bernstein--von Mises limit in
\eqref{eq:posterior-mean}.

We recall that the random vector
$\xi_t$, conditioned on $\eta_t$, has density $f(\cdot | \eta_t, \theta)$
and $\{\eta_t\}_{t\geq1}$ evolves according to a (parameter-free)
Markov kernel on $\mathcal{H}$.

\begin{assumption}
\label{assume:infinite-horizon-2}
{\rm (i)}
There exists $\epsilon >0$ such that $B_\epsilon(\theta^*)\subseteq \Theta$,
and for  all $(\xi,\eta) \in \Xi \times \mathcal{H}$, 
$f(\xi | \eta, \cdot)$
is continuously differentiable on $B_\epsilon(\theta^*)$.
{\rm (ii)}
For all $\eta \in \mathcal{H}$, 
$
\|s(\cdot| \eta,\theta)\|_2^2
f(\cdot|\eta, \theta)
$,
$\theta \in \overline{B}_{\epsilon/2}(\theta^*)$, 
is dominated by a Lebesgue-integrable function on $\Xi$,
where
\begin{equation}\label{score}
 s(\xi| \eta,\theta)
\coloneqq  
\nabla_\theta \log f(\xi| \eta,\theta).
\end{equation}
{\rm (iii)}
The policy $\pi^*$ of the original problem \eqref{bell-inf} 
is unique.
\end{assumption}

We verify \Cref{assume:infinite-horizon-2}(i)--(ii) for a regression model.

\begin{example}
Let $\Xi=\mathbb{R}^d$ and let $\mathcal{H}\subseteq\mathbb{R}^q$ be
finite. 
Suppose that $B_\epsilon(\theta^*)\subseteq\Theta$ for some
$\epsilon>0$. 
For $A\in\mathbb{R}^{d\times q}$ and $b\in\mathbb{R}^d$, let
$\theta=(A,b)\in\mathbb{R}^{d\times q}\times\mathbb{R}^d$, where the
parameter space is equipped with the norm
$\|\theta\|_2\coloneqq
(\|A\|_{\mathrm F}^2+\|b\|_2^2)^{1/2}$.
Here $\|\cdot\|_{\mathrm{F}}$ is the Frobenius norm.
 Define
$\chi(\eta,\theta)\coloneqq A\eta+b$. For a symmetric positive definite
matrix $\Sigma\in\mathbb{R}^{d\times d}$, consider
\begin{equation}\label{cont-1}
    \xi_t=\chi(\eta_t,\theta)+\varepsilon_t,
    \qquad
    \varepsilon_t\sim\mathcal{N}(0,\Sigma),
\end{equation}
where $\varepsilon_t$ is independent of $\eta_t$.
Then $\xi_t|\eta_t\sim \N(\chi(\eta_t,\theta),\Sigma)$ and
\[
f(\xi_t|\eta_t,\theta)=\frac{1}{(2\pi)^{d/2}|\Sigma|^{1/2}}
\exp\big\{-(\xi_t-\chi(\eta_t,\theta))^\top
 \Sigma^{-1}(\xi_t-\chi(\eta_t,\theta))/2\big\}.
 \]
Using the above expression of $f$, we find that
\Cref{assume:infinite-horizon-2}(i) is satisfied.
Now we verify \Cref{assume:infinite-horizon-2}(ii).
We define
$\ell(\xi,\eta,\theta)\coloneqq  \xi-A\eta-b$. Hence
\[
\nabla_b \log f(\xi|\eta,\theta)=\Sigma^{-1}\ell(\xi,\eta,\theta),
\qquad
\nabla_A \log f(\xi|\eta,\theta)=\Sigma^{-1}\ell(\xi,\eta,\theta)\,\eta^\top.
\]
Since $\overline{B}_{\epsilon/2}(\theta^*)$ is compact,
and $\mathcal{H}$ is finite, there exists a constant
$C > 0$ such that, 
for every $\eta\in\mathcal{H}$ and 
$\theta\in \overline{B}_{\epsilon/2}(\theta^*)$, 
\[
\|s(\xi\mid\eta,\theta)\|_2^2 f(\xi\mid\eta,\theta)
\le
C(1+\|\xi\|_2^2)
\exp(-\|\xi\|_2^2/C).
\]
The right-hand side is Lebesgue integrable on $\mathbb{R}^d$. Hence,
\Cref{assume:infinite-horizon-2}(ii) holds.
\end{example}

Let \Cref{assump:consistency,assume:infinite-horizon-1,assume:infinite-horizon-2}
hold.
For fixed initial values $(x_1,\eta_1)\in\mathcal{X}\times\mathcal{H}$, we define the optimal
discounted cost, 
$W^{\pi^*}(x_1,\eta_1)$, and the cumulative score up to time $T$,
$S_{\theta,T}$, by
\begin{equation}\label{funcW}
W^{\pi^*}(x_1,\eta_1)\coloneqq \sum_{t=1}^{\infty}\gamma^{t-1}
c\big(x_t,\pi^*(x_t,\eta_t),\xi_t\big),
\quad
\text{and}
\quad 
S_{\theta,T}\coloneqq \sum_{t=1}^{T}s(\xi_t|\eta_t,\theta).
\end{equation}
To connect the random total cost
\(W^{\pi^*}(x_1,\eta_1)\) to a value function, and to study how this
value depends on the model parameter, we introduce the
policy-evaluation Bellman operator associated with the fixed policy
\(\pi^*\) under parameter \(\theta\).
For $(x,\eta,\theta)\in \mathcal{X} \times \mathcal{H} \times \Theta$, 
and $V \in \mathbb{B}(\mathcal{X} \times \mathcal{H})$,
we define the Bellman operator
\begin{equation*}
[\mathcal{T}_{\theta}^{\pi^*} V](x,\eta)
\coloneqq  
\bbe_{\xi\sim f(\cdot| \eta, \theta)}
\Big[
c(x,\pi^*(x,\eta),\xi)+\gamma\,
\bbe_{\eta'|\eta}
\big[V(F(x,\pi^*(x,\eta),\xi),\eta')\big]
\Big].
\end{equation*}
Since $\gamma \in (0,1)$,
$\mathcal{T}_{\theta}^{\pi^*}$ is a $\gamma$-contraction operator.
This ensures that the fixed point equation
\(V_{\theta}^{\pi^*}=\mathcal{T}_{\theta}^{\pi^*}
V_{\theta}^{\pi^*}\) has a unique solution, which we denote by
\(V_{\theta}^{\pi^*}\).
Then, \Cref{assump:consistency,assume:infinite-horizon-1,assume:infinite-horizon-2} ensure
$$V^*=V_{\theta^*}^{\pi^*}
\quad \text{and} \quad 
V_{\theta}^{\pi^*}(x_1,\eta_1)=\mathbb{E}_{\mathbb{P}_\theta}[W^{\pi^*}(x_1,\eta_1)].$$

Our next lemma shows that the expected long-run discounted performance under the optimal policy is a differentiable function of the parameters, and it provides an explicit gradient formula.

\begin{lemma}
\label{lem:differentiability-cost}
Let \Cref{assump:consistency,assume:infinite-horizon-1,assume:infinite-horizon-2} hold, 
and let $(x_1, \eta_1) \in \mathcal{X} \times \mathcal{H}$.
Then $g(\theta) \coloneqq  V_{\theta}^{\pi^*}(x_1, \eta_1) $
is bounded and differentiable in a neighborhood of $\theta^*$ with 
$$\nabla g(\theta) = 
\mathbb{E}_{\mathbb{P}_\theta}
\bigg[
 \sum_{t=1}^{\infty}  \gamma^{t-1}
c(x_t,\pi^*(x_t,\eta_t),\xi_t) S_{\theta, t}
\bigg].$$
\end{lemma}

\begin{proof}{Proof}
Since $V_{\theta}^{\pi^*}$ solves the fixed point equation
$V_{\theta}^{\pi^*} = \mathcal{T}_{\theta}^{\pi^*} V_{\theta}^{\pi^*}$, 
$\gamma \in (0,1)$, 
and the one-stage cost $c$ is bounded, the value function $V_{\theta}^{\pi^*}$  is bounded.
Hence $g$ is bounded.
To verify the differentiability, 
we use the identity $g(\theta) = \mathbb{E}_{\mathbb{P}_\theta}[W^{\pi^*}(x_1,\eta_1)]$.
For $T \in \mathbb{N}$, we define
$\boldsymbol{c}_t \coloneqq  
c(x_t,\pi^*(x_t,\eta_t),\xi_t)$, 
$$
\boldsymbol{s}_t \coloneqq  
s(\xi_t| \eta_t,\theta), \quad 
W_T(x_1, \eta_1)\coloneqq \sum_{t=1}^{T}\gamma^{t-1}\,
\boldsymbol{c}_t, 
\quad 
\text{and}
\quad 
g_T(\theta)\coloneqq \mathbb{E}_{\mathbb{P}_\theta}[W_T(x_1, \eta_1)].
$$
We show that $g_T$ is differentiable on $B_{\varepsilon/2}(\theta^*)$
with gradient $\nabla g_T(\theta)
=\mathbb{E}_{\mathbb{P}_\theta}[W_T(x_1, \eta_1) S_{\theta,T}]$.
Since $\{\eta_t\}_{t\geq1}$ is parameter free,
the likelihood of $\{(\xi_{t},\eta_{t})\}_{1\leq t \leq T}$ factors as
$\prod_{t=1}^T f(\xi_t|\eta_t,\theta)$ times a factor independent of
$\theta$. Hence
$$
\nabla_\theta \log \Big(\prod_{t=1}^T f(\xi_t|\eta_t,\theta)\Big)
=\sum_{t=1}^T\nabla_\theta \log f(\xi_t|\eta_t,\theta)
=S_{\theta,T}.
$$
Applying Theorem~9.56 in \cite{SDR}, we obtain 
$
\nabla g_T(\theta)=\mathbb{E}_{\mathbb{P}_\theta}[W_T(x_1, \eta_1)  S_{\theta,T}]
$.
We have  $\mathbb{E}_{\mathbb{P}_\theta}[\boldsymbol{s}_{t} | \eta_{[t]},
\xi_{[t-1]}] = 0$.
For $t^\prime > t$, 
$
\mathbb{E}_{\mathbb{P}_\theta}[
\boldsymbol{c}_t \boldsymbol{s}_{t^\prime}
]
= 
\mathbb{E}_{\mathbb{P}_\theta}
[
\boldsymbol{c}_t \mathbb{E}_{\mathbb{P}_\theta}[\boldsymbol{s}_{t^\prime} | \eta_{[t^\prime]},
\xi_{[t^\prime-1]}]
]
= 0$.
We have 
$$
\nabla g_T(\theta) = \sum_{t=1}^{T}  \sum_{t^\prime=1}^{T}\gamma^{t-1}
\mathbb{E}_{\mathbb{P}_\theta}[
\boldsymbol{c}_t \boldsymbol{s}_{t^\prime}
]
= \sum_{t=1}^{T}  \gamma^{t-1}
\mathbb{E}_{\mathbb{P}_\theta}[
\boldsymbol{c}_t S_{\theta,t}
]
.$$

Since $c$ is bounded, there exists a constant $C>0$
such that $|c(x, u, \xi)| \leq C$ for all $(x, u, \xi) \in \mathcal{X} \times \mathcal{U} \times \Xi$.
Hence, for all $\theta\in \Theta$,
$$
\abs{g(\theta)-g_T(\theta)}
\le \mathbb{E}_{\mathbb{P}_\theta}\big[|W^{\pi^*}(x_1,\eta_1) -W_T(x_1, \eta_1) |\big]
\le C(1-\gamma)^{-1}\gamma^{T},
$$
implying that $g_T\to g$ uniformly on $\Theta$.

Next, we show that $\nabla g_T$ converges to $h$ uniformly on $\overline{B}_{\epsilon/4}(\theta^*)$, where
$h(\theta)\coloneqq   \sum_{t=1}^{\infty}  \gamma^{t-1}
\mathbb{E}_{\mathbb{P}_\theta}[
\boldsymbol{c}_tS_{\theta, t}
]$.
We have
$$
\|\nabla g_T(\theta) - h(\theta)\|_2
\leq 
C\sum_{t=T+1}^\infty
\gamma^{t-1}
\mathbb{E}_{\mathbb{P}_\theta}
[
\|S_{\theta, t}\|_2
]
\leq 
C \sup_{t \geq 1}\mathbb{E}_{\mathbb{P}_\theta}\big[
\|s(\xi_t| \eta_t,\theta)\|_2\big] \sum_{t=T+1}^\infty
t \gamma^{t-1}.
$$
Hence, $\nabla g_T\to h$ uniformly on $\overline{B}_{\epsilon/4}(\theta^*)$, 
yielding $\nabla g(\theta)=h(\theta)$.
\end{proof}

\begin{remark}
\label{rem:sensitivity-fixed-point}
Let \Cref{assump:consistency,assume:infinite-horizon-1,assume:infinite-horizon-2} hold. 
The differentiability result in
\Cref{lem:differentiability-cost} can also be interpreted through a Bellman identity.
Fix
\((x_1,\eta_1)\in \mathcal{X}\times\mathcal{H}\).
The Bellman equation gives
\[
V_{\theta}^{\pi^*}(x_1,\eta_1)
=
\bbe_{\xi\sim f(\cdot|\eta_1,\theta)}
\Big[
c(x_1,\pi^*(x_1,\eta_1),\xi)
+\gamma
\bbe_{\eta'|\eta_1}
\big[
V_{\theta}^{\pi^*}
(F(x_1,\pi^*(x_1,\eta_1),\xi),\eta')
\big]
\Big].
\]
\Cref{lem:differentiability-cost} justifies
differentiating this equation with respect to \(\theta\). Combined with
\(\nabla_{\theta} f(\xi|\eta_1,\theta)
=f(\xi|\eta_1,\theta)s(\xi|\eta_1,\theta)\), and
Theorem~9.56 in \cite{SDR}, we obtain
\begin{equation}\label{eq:grad-fixed-point}
\begin{aligned}
\nabla_{\theta} V_{\theta}^{\pi^*}(x_1,\eta_1)
&=
\bbe_{\xi\sim f(\cdot| \eta_1,\theta)}
\Big[
s(\xi|\eta_1,\theta)
\Big(
c(x_1,\pi^*(x_1,\eta_1),\xi)
\\
&\qquad\qquad
+\gamma
\bbe_{\eta'|\eta_1}
\big[
V_{\theta}^{\pi^*}
(F(x_1,\pi^*(x_1,\eta_1),\xi),\eta')
\big]
\Big)
\Big]
\\
&\quad
+\gamma
\bbe_{\xi\sim f(\cdot| \eta_1,\theta)}
\Big[
\bbe_{\eta'|\eta_1}
\big[
\nabla_{\theta} V_{\theta}^{\pi^*}
(F(x_1,\pi^*(x_1,\eta_1),\xi),\eta')
\big]
\Big].
\end{aligned}
\end{equation}
\end{remark}

We are ready to formulate a Bernstein--von Mises-type theorem for the
data-driven contextual optimal value. For a fixed initial values
\((x_1,\eta_1)\in \mathcal{X}\times\mathcal{H}\), the theorem establishes
asymptotic normality of
\(V_N^*(x_1,\eta_1)\) after centering at the true optimal value
\(V^*(x_1,\eta_1)\) and scaling by \(N^{1/2}\).
We define the Fisher information matrices
\begin{equation}
\label{eq:fisher-matrices}
I(\theta;\eta)
\coloneqq  
\mathbb{E}_{\xi \sim f(\cdot | \eta, \theta)}\big[
s(\xi| \eta,\theta) s(\xi| \eta,\theta)^\top
\big]
\quad 
\text{and} \quad 
I(\theta)
\coloneqq  
\sum_{h\in \HH}\nu_\eta(h) I(\theta;h).
\end{equation}

Recalling the empirical log-likelihood \(\lambda_N\) from
\eqref{eq:phi_N}, let \(\hat\theta_N\) be one of its maximizers.
To state the posterior approximation in the notation of the contextual
model, let \(\mathsf p_N^*\) denote the density of the local parameter
\(\tau=N^{1/2}(\theta-\hat\theta_N)\).

\begin{assumption}[Contextual weighted Bernstein--von Mises approximation]
\label{assump:contextual-weighted-bvm}
The matrix $I(\theta^*)$ is nonsingular,
$
N^{1/2}(\hat\theta_N-\theta^*)
\overset{\mathbb P^*}{\rightsquigarrow}
\mathcal N(0,I(\theta^*)^{-1})
$,
and, 
\begin{equation}
\label{eq:BvM-contextual-local-posterior}
\lim_{N\to\infty}
\int_{\mathbb R^s}
(1+\|\tau\|_2^2)
|\mathsf p_N^*(\tau)-\phi_I(\tau)|
\,\mathrm d\tau
=0,
\quad \mathbb P^*\text{-almost surely},
\end{equation}
where
 $\phi_I$ is  the density of
$\mathcal N(0,I(\theta^*)^{-1})$.
\end{assumption}
This is the contextual analogue of the generic Markov-chain posterior
approximation formulated in \Cref{assump:weighted-bvm}.

\begin{theorem}
\label{thm:BvM-contextual-optimal-value}
Let \Cref{assump:consistency,assume:infinite-horizon-1,assume:infinite-horizon-2,assump:contextual-weighted-bvm} hold,
let $(x_1, \eta_1) \in \mathcal{X} \times \mathcal{H}$ be fixed.
Moreover, suppose that
\begin{equation}
\label{eq:asymptotic-expansion}
N^{1/2}
\big(
V_N^*(x_1,\eta_1)
-
 \mathbb{E}_{\theta \sim \mathsf{p}_N}
\big[
V_{\theta}^{\pi^*}(x_1, \eta_1) 
\big]
\big)
\xrightarrow{\mathbb{P}^*}
0.
\end{equation}
Then
\begin{equation}
\label{eq:BvM-limit}
N^{1/2}\big(V_N^*(x_1,\eta_1)-V^*(x_1,\eta_1)\big)
\overset{\mathbb{P}^*}{\rightsquigarrow}
\mathcal{N}\Big(
0,\,
\mathbb{E}_{\mathbb{P}_{\theta^*}}[H_{\theta^*}(x_1,\eta_1)]^\top
I(\theta^*)^{-1}
\mathbb{E}_{\mathbb{P}_{\theta^*}}[H_{\theta^*}(x_1,\eta_1)]
\Big),
\end{equation}
where
\begin{equation}
H_{\theta^*}(x_1,\eta_1)
\coloneqq  
\sum_{t=1}^{\infty}\gamma^{t-1}\,
c(x_t,\pi^*(x_t,\eta_t),\xi_t)\,S_{\theta^*,t}.
\end{equation}
\end{theorem}

Before we establish \Cref{thm:BvM-contextual-optimal-value}, we comment on
its hypotheses and its assertion.
The condition \eqref{eq:asymptotic-expansion} ensures
that the asymptotic distribution of
$N^{1/2}
\big(V_N^*(x_1,\eta_1)- V^*(x_1,\eta_1)\big)$
equals that of 
\begin{equation}
\label{eq:asymptotic-equivalence}
N^{1/2}
\big(
 \mathbb{E}_{\theta \sim \mathsf{p}_N}
\big[V_{\theta}^{\pi^*}(x_1,\eta_1)
\big]
- 
V^*(x_1,\eta_1)
\big).
\end{equation}
This statement is with respect to the probability law
$\mathbb{P}^*$ of the data-generating process.
Recalling  \(V^*(x_1,\eta_1)=V_{\theta^*}^{\pi^*}(x_1,\eta_1)\)
under a correctly specified,  identifiable model,
the condition \eqref{eq:asymptotic-expansion}  requires 
that the error between the Bayesian optimal value, 
$V_N^*(x_1, \eta_1)$, and the posterior
mean of  $\theta\mapsto V_{\theta}^{\pi^*}(x_1,\eta_1)$
is asymptotically negligible at the $N^{-1/2}$ scale.
If $\pi_N^*$ is a Bayesian optimal policy, then for
each $(x,\eta) \in \mathcal{X} \times \mathcal{H}$,
\begin{equation}
\label{eq:VNVtheta-identity}
\begin{aligned}
& V_N^*(x,\eta)
-
 \mathbb{E}_{\theta \sim \mathsf{p}_N}
\big[
V_{\theta}^{\pi^*}(x, \eta) 
\big]
\\
& \quad =
\bbe_{\mathsf{p}_N| \eta}\Big[
c(x,\pi_N^*(x,\eta),\xi)+\gamma\,
\bbe_{\eta'|\eta}
\big[V_N^*(F(x,\pi_N^*(x,\eta),\xi),\eta')\big]
\Big]
\\
& \quad \quad -
\bbe_{\mathsf{p}_N| \eta}\Big[
c(x,\pi^*(x,\eta),\xi)+\gamma\,
\bbe_{\eta'|\eta}
\big[V_\theta^{\pi^*}(F(x,\pi^*(x,\eta),\xi),\eta')\big]
\Big].
\end{aligned}
\end{equation}
Hence  \(V_N^*(x,\eta)\) differs from
\(
 \mathbb{E}_{\theta \sim \mathsf{p}_N}
\big[
V_{\theta}^{\pi^*}(x, \eta) 
\big]
\)  in that:  (i) it evaluates the objective under the Bayesian optimal policy \(\pi_N^*\),
whereas the posterior mean
\(\mathbb{E}_{\theta\sim \mathsf{p}_N}[V_\theta^{\pi^*}(x,\eta)]\)
uses the true optimal policy \(\pi^*\), and (ii) it uses \(V_N^*\) in
place of \(V_\theta^{\pi^*}\) in \eqref{eq:VNVtheta-identity}. The 
asymptotic equivalence in \eqref{eq:asymptotic-expansion} 
partially resembles the finite-dimensional, single-stage
stochastic optimization setting; see eq.\ (5.24) in \cite{SDR}. A key difference
here is that \eqref{eq:VNVtheta-identity} involves two distinct value functions
on the right-hand side. 
After proving \Cref{thm:BvM-contextual-optimal-value}, we verify
\eqref{eq:asymptotic-expansion} in the finite state-action setting in
\Cref{subsec:finite-state-action-verification}.

Let us compare the asymptotic nature of \Cref{thm:consistency} with that of \Cref{thm:BvM-contextual-optimal-value}.
The posterior consistency in \Cref{thm:consistency} is a two-level asymptotic: for
$\mathbb{P}^*$-almost every realized data path (``outer'' randomness), the conditional
posterior law of  $\theta_N\sim\mathsf{P}_N$ (``inner'' randomness)
concentrates on $\Theta^*$.
In contrast, \Cref{thm:BvM-contextual-optimal-value} is a one-level statement
under $\mathbb{P}^*$: $V_N^*(x_1,\eta_1)$ and
$\mathbb{E}_{\theta\sim\mathsf{p}_N}[V_\theta^{\pi^*}(x,\eta)]$ 
already integrate out the posterior uncertainty in $\theta$ and, hence, are
functions only of the observed data $\{(\xi_i,\eta_i)\}_{i=1}^N$.
\Cref{thm:BvM-contextual-optimal-value} 
establishes asymptotics for the Bayesian optimal value, 
rather than convergence of the posterior distribution itself.

The asymptotic variance in \Cref{thm:BvM-contextual-optimal-value} involves $\mathbb{E}_{\mathbb{P}_{\theta^*}}[H_{\theta^*}(x_1,\eta_1)]$.
This vector captures the first-order sensitivity of the infinite-horizon discounted cost,
evaluated under the true optimal policy, to perturbations of the 
model. 
Moreover, the matrix $I(\theta^*)^{-1}$ in \eqref{eq:BvM-limit} describes how statistical uncertainty
about $\theta$ propagates into uncertainty about the optimal value.

\Cref{thm:BvM-contextual-optimal-value}
yields Bernstein--von Mises asymptotics for each fixed initial condition
\((x_1,\eta_1) \in \mathcal{X}\times\mathcal{H}\); obtaining limits that hold uniformly over
\(\mathcal{X}\times\mathcal{H}\) is substantially more delicate. For 
uniform asymptotics in finite-horizon stochastic optimal control without
contextual states, we refer the reader to \cite{Milz2025}.

\begin{proof}{Proof of \Cref{thm:BvM-contextual-optimal-value}}
We first specialize the generic Markov-chain formulation in
\Cref{lem:posterior-delta} to the contextual data process. Set
\(Z_t=(\xi_t,\eta_t)\), \(\mathcal Z=\Xi\times\HH\), and \(r=p\).
Since \(\eta_1\) is deterministic, the \(\theta\)-dependent initial
likelihood contribution is
$
h_1((\xi, h);\theta)= \mathbf{1}_{h=\eta_1}f(\xi|\eta_1,\theta) 
$.
For \(i\ge 2\), the transition density of \(Z_i\) given \(Z_{i-1}\) is
\[
h\big((\xi_i,\eta_i)|(\xi_{i-1},\eta_{i-1}),\theta\big)
=
\varpi_{\eta_{i-1},\eta_i}f(\xi_i|\eta_i,\theta).
\]
Consequently,
using $\sff_N$ from \eqref{post-bayes-inf},
\[
\mathsf h_N(\theta)
=
f(\xi_1|\eta_1,\theta)
\prod_{i=2}^N
\big[
\varpi_{\eta_{i-1},\eta_i}f(\xi_i|\eta_i,\theta)
\big]
=
\sff_N(\theta).
\]
Thus, because \(r=p\), the posterior density \(\mathsf r_N\) in
\Cref{lem:posterior-delta} is exactly the contextual posterior density
\(\mathsf p_N\) in \eqref{post-bayes-inf}. The factor
\(\prod_{i=2}^N\varpi_{\eta_{i-1},\eta_i}\) does not depend on
\(\theta\). Hence maximizing \(\log\mathsf h_N(\theta)\) is equivalent
to maximizing \(\lambda_N(\theta)\) from \eqref{eq:phi_N}, 
and the generic maximum likelihood
estimator coincides with the contextual maximizer \(\hat\theta_N\).
Consequently, the corresponding local posterior density
\(\mathsf r_N^*\) coincides with \(\mathsf p_N^*\). Under correct
specification, the law \(\mathfrak P^*\) of \(\{Z_t\}_{t\ge1}\) under
the true parameter \(\theta^*\) is the contextual data-generating law
\(\mathbb P^*\).

Using the score function $s$ from 
\eqref{score}, the transition density above also gives
\[
\nabla_\theta\log h(Z_2|Z_1,\theta)
=
\nabla_\theta\log f(\xi_2|\eta_2,\theta)
=
s(\xi_2|\eta_2,\theta),
\]
because the transition probability \(\varpi_{\eta_1,\eta_2}\) does not
depend on \(\theta\). Under the stationary law used in
\eqref{eq:fisher-matrix-general-setting}, the marginal distribution of
\(Z_2=(\xi_2,\eta_2)\) is given by \eqref{eq:Pstar}, that is, 
$
\mathrm dP^*(\xi,\eta)
=
\nu_\eta(\eta)f(\xi|\eta,\theta^*)\,\mathrm d\xi 
$.
Therefore, by \eqref{eq:fisher-matrix-general-setting} and
\eqref{eq:fisher-matrices},
\[
\begin{aligned}
\mathfrak I(\theta^*)
=
\mathbb E_{(\xi,\eta)\sim P^*}
\big[
s(\xi|\eta,\theta^*)s(\xi|\eta,\theta^*)^\top
\big]
=
\sum_{h\in\HH}\nu_\eta(h)
\mathbb E_{\xi\sim f(\cdot|h,\theta^*)}
\big[
s(\xi|h,\theta^*)s(\xi|h,\theta^*)^\top
\big]
=
I(\theta^*).
\end{aligned}
\]
Thus the hypotheses of \Cref{lem:posterior-delta} hold with
\(\mathsf r_N=\mathsf p_N\), \(\mathsf r_N^*=\mathsf p_N^*\),
\(\mathfrak P^*=\mathbb P^*\), and
\(\mathfrak I(\theta^*)=I(\theta^*)\). In particular,
\eqref{eq:BvM-contextual-local-posterior} is exactly
\eqref{eq:BvM-Markov-Chain}.

\Cref{lem:differentiability-cost} implies that $g(\theta) \coloneqq  V_{\theta}^{\pi^*}(x_1, \eta_1) $ is
bounded and 
differentiable at \(\theta^*\), with
$
\nabla g(\theta^*)
=
\mathbb{E}_{\mathbb{P}_{\theta^*}}
[H_{\theta^*}(x_1,\eta_1)]
$.
Applying \Cref{lem:posterior-delta}  to
\(g\) gives
\[
\begin{aligned}
N^{1/2}
\big(
\mathbb{E}_{\theta\sim\mathsf p_N}
[V_{\theta}^{\pi^*}(x_1,\eta_1)]
-
V_{\theta^*}^{\pi^*}(x_1,\eta_1)
\big)
\overset{\mathbb{P}^*}{\rightsquigarrow}
\mathcal{N}
\big(
0,
\nabla g(\theta^*)^\top
I(\theta^*)^{-1}
\nabla g(\theta^*)
\big).
\end{aligned}
\]
By correct specification and identifiability,
\(\mathbb{P}^*=\mathbb{P}_{\theta^*}\) and
\(V_{\theta^*}^{\pi^*}(x_1,\eta_1)=V^*(x_1,\eta_1)\). Hence 
\[
N^{1/2}
\big(
\mathbb{E}_{\theta\sim\mathsf p_N}
[V_{\theta}^{\pi^*}(x_1,\eta_1)]
-
V^*(x_1,\eta_1)
\big)
\overset{\mathbb{P}^*}{\rightsquigarrow}
\mathcal{N}
\big(
0,
\nabla g(\theta^*)^\top
I(\theta^*)^{-1}
\nabla g(\theta^*)
\big).
\]
Next, we decompose
\[
\begin{aligned}
N^{1/2}\big(V_N^*(x_1,\eta_1)-V^*(x_1,\eta_1)\big) 
& =
N^{1/2}
\big(
V_N^*(x_1,\eta_1)
-
\mathbb{E}_{\theta\sim\mathsf p_N}
[V_{\theta}^{\pi^*}(x_1,\eta_1)]
\big) \\
&\quad
+
N^{1/2}
\big(
\mathbb{E}_{\theta\sim\mathsf p_N}
[V_{\theta}^{\pi^*}(x_1,\eta_1)]
-
V^*(x_1,\eta_1)
\big).
\end{aligned}
\]
The first term converges to zero in \(\mathbb{P}^*\)-probability by
\eqref{eq:asymptotic-expansion}. The second term has the normal limit
identified above. Now, Slutsky's lemma 
and the gradient formula
$
\nabla g(\theta^*)
=
\mathbb{E}_{\mathbb{P}_{\theta^*}}
[H_{\theta^*}(x_1,\eta_1)]
$
give \eqref{eq:BvM-limit}.
\end{proof}

\subsection{Finite state-action SOC models}
\label{subsec:finite-state-action-verification}

The next result shows that the asymptotic equivalence
assumption \eqref{eq:asymptotic-expansion} 
in \Cref{thm:BvM-contextual-optimal-value} holds in the
finite state-action case.

\begin{proposition}
\label{prop:finite-state-action-asymptotic-expansion}
Suppose that \(\mathcal X\) and \(\U\) are finite, and suppose that
\Cref{assump:consistency,assume:infinite-horizon-1,assume:infinite-horizon-2,assump:contextual-weighted-bvm} hold. Then, for every fixed
\((x_1,\eta_1)\in\mathcal X\times\HH\),
the condition in  \eqref{eq:asymptotic-expansion} holds.
\end{proposition}

We defer the proof of \Cref{prop:finite-state-action-asymptotic-expansion} to 
\Cref{app:finite-state-action-verification}.

\paragraph{Acknowledgments}

We thank the Associate Editor and the two anonymous referees for
their careful reading and constructive comments, which helped improve the
manuscript.
We thank Xin Chen for valuable discussions and sharing relevant literature.
We used ChatGPT~5.2 and 5.5, and Microsoft Copilot (GPT 5.2) to polish parts of the manuscript.
Moreover, we used Gemini~3 Pro 
and ChatGPT~5.5
to assist with generating TikZ figures.

\appendix

\section{Proof of \Cref{lem:posterior-delta}}
\label{app:posterior-delta}

We prepare our proof of \Cref{lem:posterior-delta}. 
By the definition of the posterior local density \(\mathsf{r}_N^*\), for any posterior-integrable
function \(h\),
\begin{equation}
\label{eq:posterior-mean-h}
\mathbb{E}_{\theta\sim\mathsf{r}_N}[h(\theta)]
=
\int_{N^{1/2}(\Theta-\hat{\theta}_N)}
h\big(\hat\theta_N+N^{-1/2}\tau\big)
\mathsf{r}_N^*(\tau)\,\mathrm{d}\tau .
\end{equation}

\begin{proof}{Proof of \Cref{lem:posterior-delta}}
Set \(B_N\coloneqq N^{1/2}(\hat\theta_N-\theta^*)\). Using \eqref{eq:posterior-mean-h}, 
we have
\[
\mathbb{E}_{\theta\sim\mathsf{r}_N}[g(\theta)]
=
\int_{N^{1/2}(\Theta-\hat{\theta}_N)}
g\big(\theta^*+N^{-1/2}(B_N+\tau)\big)
\mathsf{r}_N^*(\tau)\,\mathrm{d}\tau .
\]
Using \eqref{eq:BvM-Markov-Chain}, we obtain
\begin{equation}
\label{eq:bvm-local-posterior-moments}
\int_{\mathbb{R}^s}
\tau\mathsf{r}_N^*(\tau)\,\mathrm{d}\tau
=o_{\mathfrak{P}^*}(1),
\quad \text{and} \quad 
\int_{\mathbb{R}^s}
\|\tau\|_2^2\mathsf{r}_N^*(\tau)\,\mathrm{d}\tau
=O_{\mathfrak{P}^*}(1).
\end{equation}
The first relation follows because the limiting normal distribution is
	centered, and the second follows from the weight \(1 + \|\tau\|_2^2\)
in \eqref{eq:BvM-Markov-Chain}.
Since \(g\) is differentiable at \(\theta^*\), there exists a function
\(\rho\), with \(\rho(h)\to0\) as \(h\to0\), such that
\[
g(\theta^*+h)
=
g(\theta^*)+\nabla g(\theta^*)^\top h+\rho(h)\|h\|_2 .
\]
Therefore,
\begin{equation}
\label{eq:g-taylor-expansion}
N^{1/2}
\big(
\mathbb{E}_{\theta\sim\mathsf{r}_N}[g(\theta)]-g(\theta^*)
\big)
 =
\nabla g(\theta^*)^\top B_N
+
\nabla g(\theta^*)^\top
\bigg[\int_{N^{1/2}(\Theta-\hat{\theta}_N)}
\tau\mathsf{r}_N^*(\tau)\,\mathrm{d}\tau
\bigg]
+
R_N ,
\end{equation}
where \(R_N\) is the remainder term defined by
\[
\begin{aligned}
R_N\coloneqq 
\int_{N^{1/2}(\Theta-\hat{\theta}_N)}
\Big[
N^{1/2}
\big(
g\big(\theta^*+N^{-1/2}(B_N+\tau)\big)-g(\theta^*)
\big)
-
\nabla g(\theta^*)^\top(B_N+\tau)
\Big]
\mathsf{r}_N^*(\tau)\,\mathrm{d}\tau .
\end{aligned}
\]

We show that \(R_N=o_{\mathfrak{P}^*}(1)\). Since
\(B_N=O_{\mathfrak{P}^*}(1)\), and 
using \eqref{eq:bvm-local-posterior-moments},  for every sufficiently large \(M\), the event
\[
E_N(M)\coloneqq 
\Big\{
\|B_N\|_2\le M,\quad
\int_{N^{1/2}(\Theta-\hat{\theta}_N)}
\|\tau\|_2^2\mathsf{r}_N^*(\tau)\,\mathrm{d}\tau\le M
\Big\}
\]
has probability arbitrarily close to one, uniformly for all large \(N\).
Fix \(\varepsilon>0\), and choose \(\delta>0\) such that
$\overline{B}_\delta(\theta^*) \subset \Theta$, and 
\(|\rho(h)|\le\varepsilon\) whenever \(\|h\|_2\le\delta\). On
\(E_N(M)\), we split the integral region in the remainder term  into
\[
A_N\coloneqq \{\tau \in N^{1/2}(\Theta-\hat{\theta}_N): N^{-1/2}\|B_N+\tau\|_2\leq \delta\}
\]
and its complement $A_N^c$.
On \(A_N\), we have
\[
\begin{aligned}
\int_{A_N}
|\rho(N^{-1/2}(B_N+\tau))|\|B_N+\tau\|_2
\mathsf{r}_N^*(\tau)\,\mathrm{d}\tau
\le
\varepsilon
\int_{N^{1/2}(\Theta-\hat{\theta}_N)}
\|B_N+\tau\|_2\mathsf{r}_N^*(\tau)\,\mathrm{d}\tau
\le
\varepsilon(M+M^{1/2}).
\end{aligned}
\]

Now, we bound the remainder term integral on $A_N^c$.
On \(A_N^c\), for all sufficiently large \(N\), the event \(E_N(M)\)
implies \(\|\tau\|_2> \delta N^{1/2}/2\). Hence,
on \(E_N(M)\),
\[
\begin{aligned}
\int_{A_N^c}\mathsf{r}_N^*(\tau)\,\mathrm{d}\tau
\le
\int_{\{\|\tau\|_2>\delta N^{1/2}/2\}}
\mathsf{r}_N^*(\tau)\,\mathrm{d}\tau 
\le
\frac{1}{(\delta N^{1/2}/2)^2}
\int_{N^{1/2}(\Theta-\hat{\theta}_N)}
\|\tau\|_2^2
\mathsf{r}_N^*(\tau)\,\mathrm{d}\tau 
\le
\frac{4M}{\delta^2N}.
\end{aligned}
\]
Moreover, on \(E_N(M)\),
\[
\begin{aligned}
\int_{\{\|\tau\|_2>\delta N^{1/2}/2\}}
\!\!
\|\tau\|_2\mathsf{r}_N^*(\tau)\,\mathrm{d}\tau 
\le
\frac{1}{\delta N^{1/2}/2}
\int_{\{\|\tau\|_2>\delta N^{1/2}/2\}}
\!\!
\|\tau\|_2^2\mathsf{r}_N^*(\tau)\,\mathrm{d}\tau 
\le
\frac{2}{\delta N^{1/2}}
\int_{N^{1/2}(\Theta-\hat{\theta}_N)}
\!\!
\|\tau\|_2^2\mathsf{r}_N^*(\tau)\,\mathrm{d}\tau. 
\end{aligned}
\]
and, hence, 
\[
\begin{aligned}
\int_{A_N^c}
\|B_N+\tau\|_2\mathsf{r}_N^*(\tau)\,\mathrm{d}\tau
\le
M\int_{A_N^c}\mathsf{r}_N^*(\tau)\,\mathrm{d}\tau
+
\int_{\|\tau\|_2>\delta N^{1/2}/2}
\|\tau\|_2\mathsf{r}_N^*(\tau)\,\mathrm{d}\tau
\le
\frac{4M^2}{\delta^2N}
+
\frac{2M}{\delta N^{1/2}}
=
o(1).
\end{aligned}
\]
Since \(g\) is bounded, there exists \(C>0\) such that
\(|g(\theta)|\le C\) for all \(\theta\in\Theta\). We have, 
on \(E_N(M)\),
\[
\begin{aligned}
&\int_{A_N^c}
\Big|
N^{1/2}
\big[
g\big(\theta^*+N^{-1/2}(B_N+\tau)\big)-g(\theta^*)
\big]
-
\nabla g(\theta^*)^\top(B_N+\tau)
\Big|
\mathsf{r}_N^*(\tau)\,\mathrm{d}\tau
\\
&\quad \le
2C N^{1/2}
\int_{A_N^c}\mathsf{r}_N^*(\tau)\,\mathrm{d}\tau
+
\|\nabla g(\theta^*)\|_2
\int_{A_N^c}
\|B_N+\tau\|_2\mathsf{r}_N^*(\tau)\,\mathrm{d}\tau
 =
o(1).
\end{aligned}
\]

Combining the estimates over \(A_N\) and \(A_N^c\), we have, on
\(E_N(M)\),
$
|R_N|
\leq 
\varepsilon(M+M^{1/2})+o(1)
$.
Thus, for each fixed \(M\),
\[
\limsup_{N\to\infty}|R_N|
\le
\varepsilon(M+M^{1/2})
\quad
\text{on} \quad E_N(M).
\]
Since \(\varepsilon>0\) is arbitrary, \(R_N=o(1)\) on \(E_N(M)\).
Finally, because \(M\) can be chosen so that \(E_N(M)\) has probability
arbitrarily close to one uniformly for all large \(N\), we obtain
\(R_N=o_{\mathfrak{P}^*}(1)\).

Using \eqref{eq:bvm-local-posterior-moments}, 
\(R_N=o_{\mathfrak{P}^*}(1)\), 
and  \(N^{1/2}(\hat\theta_N-\theta^\ast)\overset{\mathfrak{P}^*}{\rightsquigarrow}
\mathcal N(0,\mathfrak{I}(\theta^*)^{-1})\), 
we deduce the statistical limit in \eqref{eq:posterior-mean} from the expansion
\eqref{eq:g-taylor-expansion}.
\end{proof}

\section{Proof of \Cref{lem:posterior-second-moment}}
\label{app:posterior-second-moment}

\begin{proof}{Proof of \Cref{lem:posterior-second-moment}}
Taking \(h(\theta)=\|\theta-\theta^*\|_2^2\) in \eqref{eq:posterior-mean-h}, 
and defining 
$
B_N\coloneqq N^{1/2}(\hat\theta_N-\theta^*)
$, 
we have
\[
\begin{aligned}
\mathbb{E}_{\theta\sim\mathsf{r}_N}
[\|\theta-\theta^*\|_2^2]
=
\int_{N^{1/2}(\Theta-\hat{\theta}_N)}
\big\|
N^{-1/2}(B_N+\tau)
\big\|_2^2
\mathsf{r}_N^*(\tau)\,\mathrm{d}\tau
=
N^{-1}
\int_{N^{1/2}(\Theta-\hat{\theta}_N)}
\|B_N+\tau\|_2^2
\mathsf{r}_N^*(\tau)\,\mathrm{d}\tau .
\end{aligned}
\]
Hence
\[
\begin{aligned}
\mathbb{E}_{\theta\sim\mathsf{r}_N}
[\|\theta-\theta^*\|_2^2]
\le
2N^{-1}\|B_N\|_2^2
+
2N^{-1}
\int_{N^{1/2}(\Theta-\hat{\theta}_N)}
\|\tau\|_2^2
\mathsf{r}_N^*(\tau)\,\mathrm{d}\tau.
\end{aligned}
\]
The central limit theorem implies that $B_N$ is stochastically bounded, 
and \eqref{eq:BvM-Markov-Chain} 
ensures
$
\int_{N^{1/2}(\Theta-\hat{\theta}_N)}
\|\tau\|_2^2
\mathsf{r}_N^*(\tau)\,\mathrm{d}\tau
$
is stochastically bounded. We obtain the assertion.
\end{proof}

\section{Proof of \Cref{prop:finite-state-action-asymptotic-expansion}}
\label{app:finite-state-action-verification}

We now prepare the proof of
\Cref{prop:finite-state-action-asymptotic-expansion}. Set
\(\mathcal Z\coloneqq \mathcal X\times\HH\). For \(z=(x,\eta)\),
\(z'=(x',h)\), and \(u\in\U\), define
\[
\ell_\theta(z,u)
\coloneqq 
\mathbb{E}_{\xi\sim f(\cdot|\eta,\theta)}
[c(x,u,\xi)]
\quad \text{and} \quad 
\mathsf{G}_{\theta}(z'|z,u)
\coloneqq 
\varpi_{\eta,h}
\mathbb{E}_{\xi\sim f(\cdot|\eta,\theta)}
[
\mathbf 1_{\{F(x,u,\xi)=x'\}}
].
\]
Let
\(\ell_N\coloneqq \mathbb{E}_{\theta\sim\mathsf p_N}[\ell_\theta]\) and
\(\mathsf{G}_N\coloneqq \mathbb{E}_{\theta\sim\mathsf p_N}[\mathsf{G}_{\theta}]\). We use the notation
$
\|\mathsf{G}\|_\infty
\coloneqq 
\max_{z\in \mathcal{Z},u \in \mathcal{U}}\sum_{z'\in\mathcal Z}|\mathsf{G}(z'|z,u)|
$.

\begin{lemma}
\label{lem:finite-primitive-concentration}
Under the hypotheses of
\Cref{prop:finite-state-action-asymptotic-expansion},
$
\mathbb{E}_{\theta\sim\mathsf p_N}
\big[
\|\ell_\theta-\ell_{\theta^*}\|_\infty^2
+
\|\mathsf{G}_{\theta}-\mathsf{G}_{\theta^*}\|_\infty^2
\big]
=
O_{\mathbb{P}^*}(N^{-1})
$.
Consequently,
\(\|\ell_N-\ell_{\theta^*}\|_\infty
+\|\mathsf{G}_N-\mathsf{G}_{\theta^*}\|_\infty
=O_{\mathbb{P}^*}(N^{-1/2})\).
\end{lemma}

\begin{proof}{Proof}
Let \(\Upsilon\) be the finite collection of bounded functions on \(\Xi\)
consisting of \(c(x,u,\cdot)\) and
\(\mathbf 1_{\{F(x,u,\cdot)=x'\}}\), with
\(z,z'\in\mathcal Z\) and \(u\in\U\). For
\(\upsilon\in\Upsilon\) and \(\eta\in\HH\), define
\(m_{\upsilon,\eta}(\theta)\coloneqq 
\mathbb{E}_{\xi\sim f(\cdot|\eta,\theta)}
[
\upsilon(\xi)
]
\).
By \Cref{assume:infinite-horizon-2}(i)--(ii), 
Theorem~9.56 in \cite{SDR}
gives, for
\(\theta\in\overline B_{\epsilon/2}(\theta^*)\),
\[
\nabla_\theta m_{\upsilon,\eta}(\theta)
=
\mathbb{E}_{\xi\sim f(\cdot|\eta,\theta)}
[\upsilon(\xi)s(\xi|\eta,\theta)] ,
\]
where the score function $s$ is defined in \eqref{score};
see also the likelihood-ratio
differentiation argument used in the proof of
\Cref{lem:differentiability-cost}.
Hence
\[
\|\nabla_\theta m_{\upsilon,\eta}(\theta)\|_2
\le
\|\upsilon\|_\infty
\mathbb{E}_{\xi\sim f(\cdot|\eta,\theta)}
[\|s(\xi|\eta,\theta)\|_2^2]^{1/2}.
\]
Since \(\HH\) and \(\Upsilon\) are finite, these derivatives are uniformly
bounded on \(\overline B_{\epsilon/2}(\theta^*)\). Hence, for some
\(L<\infty\), \(\|\ell_\theta-\ell_{\theta^*}\|_\infty
+\|\mathsf{G}_{\theta}-\mathsf{G}_{\theta^*}\|_\infty
\le L\|\theta-\theta^*\|_2\) whenever
\(\theta\in\overline B_{\epsilon/2}(\theta^*)\).
Since the costs and transition probabilities are bounded, there is
\(C<\infty\) such that, for all \(\theta\in\Theta\),
\[
\begin{aligned}
\|\ell_\theta-\ell_{\theta^*}\|_\infty^2
+
\|\mathsf{G}_{\theta}-\mathsf{G}_{\theta^*}\|_\infty^2  
\le
C\|\theta-\theta^*\|_2^2
+
C\mathbf 1_{\{\|\theta-\theta^*\|_2>\epsilon/2\}} .
\end{aligned}
\]
Taking posterior expectations and using Markov's inequality yields
\[
\mathbb{E}_{\theta\sim\mathsf p_N}
\big[
\|\ell_\theta-\ell_{\theta^*}\|_\infty^2
+
\|\mathsf{G}_{\theta}-\mathsf{G}_{\theta^*}\|_\infty^2
\big]
\le
C'
\mathbb{E}_{\theta\sim\mathsf p_N}
[\|\theta-\theta^*\|_2^2].
\]
By \Cref{lem:posterior-second-moment}, applied to
\(Z_t=(\xi_t,\eta_t)\), the right-hand side is
\(O_{\mathbb{P}^*}(N^{-1})\). The second assertion follows from
Jensen's inequality.
\end{proof}

For the optimal policy \(\pi^*\), we define
\(r_\theta(z)\coloneqq \ell_\theta(z,\pi^*(z))\) and
\(\mathsf{M}_\theta(z'|z)\coloneqq  \mathsf{G}_{\theta}(z'|z,\pi^*(z))\). Let
\(r_N\coloneqq  \mathbb{E}_{\theta\sim\mathsf p_N}[r_\theta]\) and
\(\mathsf{M}_N\coloneqq  \mathbb{E}_{\theta\sim\mathsf p_N}[\mathsf{M}_\theta]\). Then, 
we have
\(V_\theta^{\pi^*}=r_\theta+\gamma \mathsf{M}_\theta V_\theta^{\pi^*}\),
where we recall $V_\theta^{\pi^*}$ from \Cref{subsect:asymptotics}.
Now, let $\overline{V}_N^{\pi^*}$ be the unique solution to
\(\overline{V}_N^{\pi^*}=r_N+\gamma \mathsf{M}_N\overline{V}_N^{\pi^*}\).

\begin{lemma}
\label{lem:finite-fixed-policy-posterior-average}
Under the hypotheses of
\Cref{prop:finite-state-action-asymptotic-expansion}, 
$
\big\|
\overline{V}_N^{\pi^*}
-
\mathbb{E}_{\theta\sim\mathsf p_N}
[V_\theta^{\pi^*}]
\big\|_\infty
=
O_{\mathbb{P}^*}(N^{-1})
$.
\end{lemma}

\begin{proof}{Proof}
Taking posterior expectation in \(V_\theta^{\pi^*}=r_\theta+\gamma \mathsf{M}_\theta V_\theta^{\pi^*}\) 
and subtracting
gives, with
\(D_N\coloneqq  \overline{V}_N^{\pi^*}
-\mathbb{E}_{\theta\sim\mathsf p_N}[V_\theta^{\pi^*}]\),
\[
D_N
=
-\gamma
(I-\gamma \mathsf{M}_N)^{-1}
\mathbb{E}_{\theta\sim\mathsf p_N}
\Big[
(\mathsf{M}_\theta-\mathsf{M}_N)
\big(
V_\theta^{\pi^*}
-
\mathbb{E}_{\vartheta\sim\mathsf p_N}
[V_\vartheta^{\pi^*}]
\big)
\Big].
\]
Since \(\mathsf{M}_N\) is a transition matrix,
\(\|(I-\gamma \mathsf{M}_N)^{-1}\|_\infty\le (1-\gamma)^{-1}\). Thus, by the
Cauchy--Schwarz inequality, \(\|D_N\|_\infty
\le \gamma(1-\gamma)^{-1}A_N^{1/2}B_N^{1/2}\), where
\[
A_N\coloneqq  
\mathbb{E}_{\theta\sim\mathsf p_N}
[\|\mathsf{M}_\theta-\mathsf{M}_N\|_\infty^2]
\quad 
\text{and}
\quad 
B_N\coloneqq  
\mathbb{E}_{\theta\sim\mathsf p_N}
\left[
\left\|
V_\theta^{\pi^*}
-
\mathbb{E}_{\vartheta\sim\mathsf p_N}
[V_\vartheta^{\pi^*}]
\right\|_\infty^2
\right].
\]

By \Cref{lem:finite-primitive-concentration},
\(A_N\le
4\mathbb{E}_{\theta\sim\mathsf p_N}
[\|\mathsf{M}_\theta-\mathsf{M}_{\theta^*}\|_\infty^2]
=O_{\mathbb{P}^*}(N^{-1})\). Moreover,
\[
V_\theta^{\pi^*}-V_{\theta^*}^{\pi^*}
=
(I-\gamma \mathsf{M}_\theta)^{-1}
\big[
r_\theta-r_{\theta^*}
+
\gamma(\mathsf{M}_\theta-\mathsf{M}_{\theta^*})V_{\theta^*}^{\pi^*}
\big].
\]
Since \(\|(I-\gamma \mathsf{M}_\theta)^{-1}\|_\infty\le(1-\gamma)^{-1}\),
there is \(C<\infty\) such that
\[
\|V_\theta^{\pi^*}-V_{\theta^*}^{\pi^*}\|_\infty
\le
C
\big(
\|r_\theta-r_{\theta^*}\|_\infty
+
\|\mathsf{M}_\theta-\mathsf{M}_{\theta^*}\|_\infty
\big).
\]
Using \Cref{lem:finite-primitive-concentration} again gives
\[
B_N
\le
4\mathbb{E}_{\theta\sim\mathsf p_N}
[\|V_\theta^{\pi^*}-V_{\theta^*}^{\pi^*}\|_\infty^2]
=
O_{\mathbb{P}^*}(N^{-1}).
\]
Therefore, \(\|D_N\|_\infty=O_{\mathbb{P}^*}(N^{-1})\).
\end{proof}

\begin{proof}{Proof of
\Cref{prop:finite-state-action-asymptotic-expansion}}
We define
$d_N\coloneqq  
\|\ell_N-\ell_{\theta^*}\|_\infty
+
\|\mathsf{G}_N-\mathsf{G}_{\theta^*}\|_\infty $.
 \Cref{lem:finite-primitive-concentration} ensures
\(d_N=O_{\mathbb P^*}(N^{-1/2})\).
We recall that \(V^*=V_{\theta^*}^{\pi^*}\). Since
\(\overline{V}_N^{\pi^*}=r_N+\gamma \mathsf{M}_N\overline{V}_N^{\pi^*}\) and
\(V^*=r_{\theta^*}+\gamma \mathsf{M}_{\theta^*}V^*\), we have
\[
\begin{aligned}
\overline{V}_N^{\pi^*}-V^*
&=
\gamma \mathsf{M}_N(\overline{V}_N^{\pi^*}-V^*)
+
r_N-r_{\theta^*}
+
\gamma(\mathsf{M}_N-\mathsf{M}_{\theta^*})V^* .
\end{aligned}
\]
Because \(\mathsf{M}_N\) is a transition matrix,
\[
\|\overline{V}_N^{\pi^*}-V^*\|_\infty
\le
(1-\gamma)^{-1}
\big(
\|r_N-r_{\theta^*}\|_\infty
+
\gamma\|(\mathsf{M}_N-\mathsf{M}_{\theta^*})V^*\|_\infty
\big).
\]
Moreover, \(\|r_N-r_{\theta^*}\|_\infty
\le \|\ell_N-\ell_{\theta^*}\|_\infty\), and
$
\|(\mathsf{M}_N-\mathsf{M}_{\theta^*})V^*\|_\infty
\le
\|V^*\|_\infty
\|\mathsf{G}_N-\mathsf{G}_{\theta^*}\|_\infty 
$.
Hence, for some constant \(C_V<\infty\),
\begin{equation}
\label{eq:overlineVNbound}
\|\overline{V}_N^{\pi^*}-V^*\|_\infty
\le
C_Vd_N .
\end{equation}

For \(z\in\mathcal Z\) and \(u\in\U\), we define
\[
\bar{Q}_N(z,u)
\coloneqq
\ell_N(z,u)
+
\gamma
\sum_{z'\in\mathcal Z}
\mathsf{G}_N(z'|z,u)\overline{V}_N^{\pi^*}(z')
\quad \text{and} \quad 
Q^*(z,u)
\coloneqq  
\ell_{\theta^*}(z,u)
+
\gamma
\sum_{z'\in\mathcal Z}
\mathsf{G}_{\theta^*}(z'|z,u)V^*(z').
\]
Using the bound in \eqref{eq:overlineVNbound}, there 
exists \(C_Q\in[1,\infty)\) such that
\[
\max_{z\in \mathcal{Z},u \in \mathcal{U}}|\bar{Q}_N(z,u)-Q^*(z,u)|
\le
C_Qd_N .
\]

Since \(\pi^*\) is the unique Bellman minimizer at every state and
\(\mathcal Z\) and \(\mathcal U\) are finite, the set
$
\{(z,u)\in\mathcal Z\times\mathcal U:u\ne\pi^*(z)\}
$
is finite. If it is the empty set, then there is no
alternative action, and hence \(V_N^*=\overline{V}_N^{\pi^*}\) for every
\(N\). Otherwise, the  gap
\[
\Delta
\coloneqq
\min_{(z,u)\in\mathcal Z\times\mathcal U:u\ne\pi^*(z)}
\{Q^*(z,u)-Q^*(z,\pi^*(z))\}
\]
is strictly positive. For this case, let
$
E_N\coloneqq \{d_N<\Delta/(3C_Q)\}
$.
On \(E_N\), for every
\(z\in\mathcal Z\) and \(u\ne\pi^*(z)\),
\[
\begin{aligned}
\bar Q_N(z,u)-\bar Q_N(z,\pi^*(z))
\ge
Q^*(z,u)-Q^*(z,\pi^*(z))-2C_Qd_N  
\ge
\Delta-2C_Qd_N
>0 .
\end{aligned}
\]
Thus, on \(E_N\), the policy \(\pi^*\) minimizes
\(\bar Q_N(z,\cdot)\) for every \(z\in\mathcal Z\). Therefore
the fixed-policy value \(\overline{V}_N^{\pi^*}\) satisfies the Bayesian
optimal Bellman equation. By uniqueness of the fixed point of the
Bayesian Bellman operator, we have
$V_N^*=\overline{V}_N^{\pi^*}$ on $E_N$. Since
\(d_N=O_{\mathbb P^*}(N^{-1/2})\), we have
\(\mathbb P^*(E_N)\to1\).
Combined with  \Cref{lem:finite-fixed-policy-posterior-average},
we obtain  \eqref{eq:asymptotic-expansion}.
\end{proof}

\bibliography{ContextBayes-arXiv-v2.bbl}

\end{document}